\documentclass{article}

\usepackage[greek, english]{babel}
\usepackage{amsmath}
\usepackage{amssymb}
\usepackage{eurosym}
\usepackage{epsf}
\usepackage{epsfig}
\usepackage{tikz}
\usetikzlibrary{arrows}

\newfont{\gothic}{eufm10}

   \def\C{{\CC}} \def\CC{{\mathbb{C}}}
\def\Z{{\mathbb{Z}}}                   \def\R{{\RR}} 
\def\RR{{\mathbb{R}}}                \def\H{{\HH}}
\def\HH{{\mathbb{H}}}

        \newtheorem{theorem}{Theorem}[section]
\newtheorem{lemma}[theorem]{Lemma}
\newtheorem{proposition}[theorem]{Proposition}
\newtheorem{corollary}[theorem]{Corollary}
\newtheorem{definition1}[theorem]{Definition}
\newenvironment{definition}{\begin{definition1}\rm}{\hfill $\triangle$\end{definition1}}
\newenvironment{proof}{\addvspace\baselineskip\noindent{\it
Proof:}\quad}{\hspace*{\fill}         $\Box$\par\addvspace\baselineskip}
\newenvironment{proofof}[1]{\addvspace\baselineskip\noindent{\it Proof}\quad}{\hspace*{\fill} $\Box$\par\addvspace\baselineskip}

\newtheorem{remark1}[theorem]{Remark}
\newenvironment{remark}{\begin{remark1}\rm}{\hfill $\triangle$\end{remark1}}

\newtheorem{example1}[theorem]{Example}
\newenvironment{example}{\begin{example1}\rm}{\hfill $\triangle$\end{example1}}

\def\barray{\begin{eqnarray*}}             \def\earray{\end{eqnarray*}}
\def\beq{\begin{equation}} \def\eeq{\end{equation}}

\makeatletter \title{Coupled cell networks and their hidden symmetries}
\author{Bob Rink\thanks{Department of Mathematics, VU University Amsterdam, The Netherlands, {\tt b.w.rink@vu.nl}.} \ and Jan Sanders\thanks{Department of Mathematics, VU University Amsterdam, The Netherlands, {\tt jan.sanders.a@gmail.com}.}
}
\begin{document}  \hyphenation{boun-da-ry mo-no-dro-my sin-gu-la-ri-ty ma-ni-fold ma-ni-folds re-fe-rence se-cond se-ve-ral dia-go-na-lised con-ti-nuous thres-hold re-sul-ting fi-nite-di-men-sio-nal ap-proxi-ma-tion pro-per-ties ri-go-rous mo-dels mo-no-to-ni-ci-ty pe-ri-o-di-ci-ties mi-ni-mi-zer mi-ni-mi-zers know-ledge ap-proxi-mate pro-per-ty poin-ting ge-ne-ra-li-za-tion ge-ne-ral re-pre-sen-ta-tions equi-variance equi-variant Equi-variance Choo-sing to-po-lo-gy brea-king}
 
\newcommand{\X}{\mathbb{X}}

\newcommand{\p}{\partial}
\maketitle
\noindent 
\abstract{\noindent Dynamical systems with a coupled cell network structure can display synchronous solutions, spectral degeneracies and anomalous bifurcation behavior. We explain these phenomena here for homogeneous networks, by showing that every homogeneous network dynamical system admits a semigroup of hidden symmetries. The synchronous solutions lie in the symmetry spaces of this semigroup and the spectral degeneracies of the network are determined by its indecomposable representations. Under a condition on the semigroup representation, we prove that a one-parameter synchrony breaking steady state bifurcation in a coupled cell network must generically occur along an absolutely indecomposable subrepresentation. We conclude with a classification of generic one-parameter bifurcations in monoid networks with two or three cells.}                 

\section{Introduction}
Coupled cell networks arise abundantly in the sciences. They vary from discrete particle models, electrical circuits and Josephson junction arrays to the world wide web, power grids, food webs and neuronal networks. 
Throughout the last decade, an extensive mathematical theory has been developed for the study of dynamical systems with a network structure \cite{field}, \cite{curious}, \cite{golstew}, \cite{stewartnature}, \cite{pivato}. 
In these network dynamical systems, the evolution of a constituent or ``cell'' is determined by the states of certain particular other cells. 

In this paper, we shall study the dynamics of homogeneous coupled cell networks. This dynamics is determined by a system of ordinary differential equations of the form 
 \begin{align}\label{diffeqnintro}
 \dot x_i =  f(x_{\sigma_1(i)}, \ldots, x_{\sigma_n(i)}) \ \mbox{for} \ 1\leq i\leq N.
\end{align}
In these equations of motion, the evolution of the state variable $x_i$ is only determined by the values of $x_{\sigma_1(i)}, \ldots, x_{\sigma_n(i)}$.  
The functions
$$\sigma_1, \ldots, \sigma_n: \{1, \ldots, N\} \to \{1, \ldots, N\}$$
should therefore be thought of as the network that decides which cells influence which cells.

A network structure can have a nontrivial impact on the behavior of a dynamical system. For instance, the network architecture of (\ref{diffeqnintro}) may force it to admit synchronous or partially synchronous solutions, cf. \cite{antonelli2},  \cite{romano},  \cite{golstew3}, \cite{torok}, \cite{stewart1}, \cite{pivato}, \cite{wang}. It has also been observed that the network structure of a dynamical system can influence its bifurcations. In fact, bifurcation scenarios that are unheard of in dynamical systems without any special structure, can occur generically in certain networks \cite{bifurcations}, \cite{anto4}, \cite{dias}, \cite{elmhirst}, \cite{krupa}, \cite{pivato2}, \cite{claire2}, \cite{synbreak}. In particular, its network structure can force the linearization of (\ref{diffeqnintro}) at a (partially) synchronous equilibrium to have eigenvalues with high multiplicity \cite{synbreak2}, \cite{leite}, \cite{feedforwardRinkSanders}. This in turn influences the solutions and bifurcations that can occur near such an equilibrium.
 
Attempts to understand this degenerate behaviour of networks have invoked the {\it groupoid formalism} of Golubitsky and Stewart et al. \cite{curious}, \cite{golstew}, \cite{stewartnature}, \cite{pivato} and more recently also the language of category theory \cite{deville}. In this paper, we propose another explanation though, inspired by the remark that invariant subspaces and spectral degeneracies are often found in dynamical systems with symmetry, cf. \cite{field4}, \cite{perspective}, \cite{golschaef2}. 

It is natural to ask whether symmetries explain the dynamical degeneracies of coupled cell networks and it has been conjectured that in general they do not \cite{golstew}. We nevertheless show in this paper that every homogeneous coupled cell network has hidden symmetries. More precisely, it turns out that equations (\ref{diffeqnintro}) are conjugate to another coupled cell network that admits a semigroup of symmetries. We call this latter network the {\it fundamental network} of (\ref{diffeqnintro}) and it is given by equations of the form
\begin{align}\label{fundamentalintro}
\dot X_j = f(X_{\widetilde \sigma_1(j)}, \ldots, X_{\widetilde \sigma_{n'}(j)})\ \mbox{for} \ 1\leq j \leq n'.
\end{align}
We will show how to compute the symmetries of (\ref{fundamentalintro}) from the network maps $\sigma_1, \ldots, \sigma_n$ and note that some of these symmetries may be represented by noninvertible transformations. 
 
It will moreover be shown that if the semigroup of symmetries is a monoid (i.e. if it contains a unit), then the fundamental network (\ref{fundamentalintro}) is completely characterized by its symmetries. This means that all the degenerate phenomena that occur in the fundamental network are due to symmetry, including the existence of synchronous solutions and the occurrence of unfamiliar bifurcations. The characterization of the fundamental network as an equivariant dynamical system is of great practical interest, because it allows us to understand the structure of equations (\ref{diffeqnintro}) and (\ref{fundamentalintro}) much better. For example, with the help of representation theory we are able to classify the generic one-parameter synchrony breaking steady state bifurcations that can be found in fundamental networks with two or three cells. We also explain how our knowledge of the fundamental network helps us understanding the behavior of the original network (\ref{diffeqnintro}). 

The remainder of this paper is organized as follows. In Section \ref{secsemigroup} we introduce the semigroup associated to the equations of motion (\ref{diffeqnintro}) and we recall some results from \cite{CCN} on semigroup coupled cell networks. In Section \ref{sechidden} we relate equations (\ref{diffeqnintro}) to the fundamental network (\ref{fundamentalintro}) and we show that the latter is equivariant under the action of the semigroup introduced in Section \ref{secsemigroup}. In Section \ref{secrepresentations} we present and prove some well-known facts from the representation theory of semigroups. We apply this theory in Sections \ref{seclyapunovschmidt} and \ref{secgeneric}, where we build a framework for the bifurcation theory of fundamental networks. More precisely, in Section \ref{seclyapunovschmidt} we introduce a variant of the method of Lyapunov-Schmidt reduction to investigate steady state bifurcations in differential equations with a semigroup of symmetries. In Section \ref{secgeneric} we then prove (under a certain condition on the semigroup) that a generic synchrony breaking steady state bifurcation in a one-parameter family of semigroup symmetric differential equations takes place along an absolutely indecomposable representation of the semigroup. We apply this result in Section \ref{sectwoorthree} to classify the generic co-dimension one synchrony breaking steady state bifurcations that can occur in monoid networks with two or three cells.

 \subsubsection*{Acknowledgement}
 The authors would like to thank Andr\'e Vanderbauwhede for very useful discussions, suggestions and explanations.

\section{Semigroup networks}\label{secsemigroup}
In this section we make some basic definitions and summarize some results from \cite{CCN}. Dynamical systems with a coupled cell network structure can be determined in various ways \cite{field}, \cite{golstew}, \cite{torok}, \cite{pivato}, but in this paper we describe it by means of a collection of distinct maps
$$\Sigma=\{\sigma_1, \ldots, \sigma_n\}\ \mbox{with} \ \sigma_1,\ldots, \sigma_n: \{1, \ldots, N\}\to\{1,\ldots, N\}\, .$$
The collection $\Sigma$ has the interpretation of a network with $1\leq N < \infty$ cells. These cells can be thought of as the vertices of a directed multigraph in which vertex $1\leq i\leq N$ receives inputs from respectively the vertices $\sigma_1(i), \ldots, \sigma_n(i)$. The idea is that the state of cell $1\leq i\leq N$ is determined by a variable $x_i$ that takes values in a vector space $V$ and that the evolution of cell $i$ is determined only by the states of the cells that act as its inputs. With this in mind, we make the following definition:
\begin{definition}\label{networkdefinition}
Let $\Sigma=\{\sigma_1, \ldots, \sigma_n\}$ be a collection of $n$ distinct maps on $N$ elements, $V$ a finite dimensional real vector space and $f: V^n\to V$ a smooth function. Then we define
\begin{align}\label{networkvectorfield}
\gamma_f:V^N\to V^N \ \mbox{by}\ (\gamma_f)_i(x):=f(x_{\sigma_1(i)}, \ldots, x_{\sigma_n(i)})\  \mbox{for}\ 1 \leq i \leq N\, . 
\end{align}
Depending on the context, we will say that $\gamma_f$ is a {\it homogeneous coupled cell network map} or a {\it homogeneous coupled cell network vector field} subject to $\Sigma$.
\end{definition}
Indeed, the coupled cell network vector field $\gamma_f$ defines a dynamical system in which the evolution of the state of cell $i$ is determined by the states of cells $\sigma_1(i), \ldots, \sigma_n(i)$, namely
$$\dot x(t) = \gamma_f(x(t))\, .$$ 
One can also view $\gamma_f$ as a map rather than a vector field and study the discrete dynamics $$x^{(n+1)}=\gamma_f(x^{(n)})\, .$$ 
\begin{example}\label{running}
Our running example of a network dynamical system will consist of $N=3$ cells and $n=3$ inputs per cell. In fact, let us choose 
$$\sigma_1[123]:=[123], \sigma_2[123]:=[121]\ \mbox{and} \ \sigma_3[123]:=[111]\, .$$ 
Then the coupled cell network maps subject to $\Sigma:=\{\sigma_1, \sigma_2, \sigma_3\}$ are of the form 
$$\gamma_f(x_1, x_2, x_3) = \left(f(x_1, x_1, x_1), f(x_2, x_2, x_1), f(x_3, x_1, x_1) \right)\, . $$
The corresponding network differential equations are
$$\begin{array}{ll} \dot x_1 =&  f(x_1, x_1, x_1) \\ \dot x_2 =&  f(x_2, x_2, x_1) \\ \dot x_3 =&  f(x_3, x_1, x_1)\end{array} \, .$$
A graphical representation of the networks maps $\sigma_1, \sigma_2, \sigma_3$ is given in Figure \ref{pict1}.
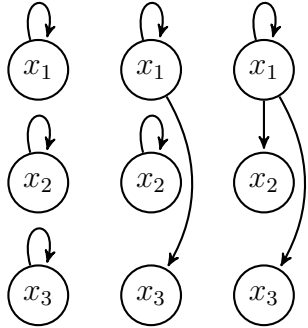
\begin{figure}[ht]\renewcommand{\figurename}{\rm \bf \footnotesize Figure}
\centering
\begin{tikzpicture}[->,>=stealth',shorten >=1pt,auto,node distance=1.5cm,
                    thick,main node/.style={circle,draw,font=\sffamily\large\bfseries}]

  \node[main node] (1) {$x_1$};
  \node[main node] (2) [below of=1] {$x_2$};
  \node[main node] (3) [below of=2] {$x_3$};

   \node[main node] (4) [right of=1] {$x_1$};
  \node[main node] (5) [below of=4] {$x_2$};
  \node[main node] (6) [below of=5] {$x_3$};
  \node[main node] (7) [right of =4] {$x_1$};
  \node[main node] (8) [below  of=7] {$x_2$};
  \node[main node] (9) [below of=8] {$x_3$};

  \path[every node/.style={font=\sffamily\small}]
(1) edge [loop above] node {} (1)
(2) edge [loop above] node {} (2)
(3) edge [loop above] node {} (3)
    
(4) edge [loop above] node {} (4)
(5) edge [loop above] node {} (5)
(4) edge [bend left] node {} (6)

(7) edge [loop above] node {} (7)
(7) edge node {} (8)
(7) edge [bend left] node {} (9)
    ;

\end{tikzpicture}
\caption{\footnotesize {\rm The collection $\{\sigma_1, \sigma_2, \sigma_3\}$ depicted as a directed multigraph.}}
\label{pict1}
\end{figure}

\end{example}
A technical problem that occurs when studying network dynamical systems is that the composition $\gamma_f\circ \gamma_g$ (or the infinitesimal composition $[\gamma_f, \gamma_g]$, the Lie bracket) of two coupled cell network maps need not be a coupled cell network map with the same network structure. This problem was addressed in \cite{CCN}, where we formulated a condition on a network that guarantees that this problem does not arise. Let us recall this condition here:
\begin{definition}
 We say that $\Sigma=\{\sigma_1, \ldots, \sigma_n\}$ is a {\it semigroup} if all its elements are distinct and if for all $1\leq j_1, j_2\leq n$ there is a $1\leq j_3\leq n$ such that 
$\sigma_{j_1}\circ \sigma_{j_2}= \sigma_{j_3}$.
 \end{definition}
 \begin{example}\label{comptable}
Recall our running Example \ref{running}. The collection $\Sigma=\{\sigma_1, \sigma_2, \sigma_3\}$ forms an abelian semigroup. Indeed, one checks that the composition table of these maps is given by:
$$\begin{array}{c|ccc} \circ & \sigma_1 & \sigma_2 & \sigma_3\\ \hline 
\sigma_1 & \sigma_1 & \sigma_2  & \sigma_3\\
\sigma_2 & \sigma_2 & \sigma_2 & \sigma_3 \\
\sigma_3 & \sigma_3 & \sigma_3 & \sigma_3
\end{array} \, .$$
\end{example}
The relevance of semigroup networks is illustrated by the following theorem. It is one of the main results in \cite{CCN} and we omit the proof here.
\begin{theorem}\label{closed}
When $\Sigma=\{\sigma_1, \ldots, \sigma_n\}$ is a semigroup, then the collection
$$\{\gamma_f\, |\, f:V^n\to V\ \mbox{smooth}\}$$
is closed under taking compositions and Lie brackets. 
\end{theorem}
\begin{example}
Recall that our running Example \ref{running} is a semigroup network. When 
\begin{align}\nonumber  
\gamma_f(x_1, x_2, x_3) & = \left(f(x_1, x_1, x_1), f(x_2, x_2, x_1), f(x_3, x_1, x_1) \right) \ \mbox{and} \\ \nonumber 
\gamma_g(x_1, x_2, x_3) & = \left(g(x_1, x_1, x_1), g(x_2, x_2, x_1), g(x_3, x_1, x_1) \right) 
\end{align}
are two coupled cell networks subject to $\Sigma$, then one computes that
$$(\gamma_f\circ \gamma_g)(x_1, x_2, x_3) = \left( \begin{array}{c} f(g(x_1, x_1, x_1), g(x_1, x_1, x_1), g(x_1, x_1, x_1)) \\  f(g(x_2, x_2, x_1), g(x_2, x_2, x_1), g(x_1, x_1, x_1))\\ f(g(x_3, x_1, x_1), g(x_1, x_1, x_1), g(x_1, x_1, x_1)) \end{array} \right) \, .$$ 
As anticipated by Theorem \ref{closed}, this shows that $\gamma_f\circ \gamma_g = \gamma_h$ with
$$h(X_1, X_2, X_3) = f(g(X_1, X_2, X_3), g(X_2, X_2, X_3), g(X_3, X_3, X_3))\, .$$
A similar computation shows that  $[\gamma_f, \gamma_g] := D\gamma_f\cdot \gamma_g-D\gamma_g\cdot \gamma_f= \gamma_{i}$ with
\begin{align} i(X_1, X_2, X_3) & = \nonumber D_1f(X_1, X_2, X_3)\cdot g(X_1, X_2, X_3) - D_1g(X_1, X_2, X_3)\cdot f(X_1, X_2, X_3) \\ \nonumber &
 + D_2f(X_1, X_2, X_3)\cdot g(X_2, X_2, X_3) - D_2g(X_1, X_2, X_3)\cdot f(X_2, X_2, X_3) \\ \nonumber &
 + D_3f(X_1, X_2, X_3)\cdot g(X_3, X_3, X_3) - D_3g(X_1, X_2, X_3)\cdot f(X_3, X_3, X_3)\, . 
\end{align}
\end{example}
Theorem \ref{closed} means that the class of semigroup network dynamical systems is a natural one to work with. 
It was shown for example in \cite{CCN} that near a dynamical equilibrium, the local normal form of a semigroup network vector field is a network vector field with the very same semigroup network structure. From the point of view of local dynamics and bifurcation theory, semigroup networks are thus very useful.

An arbitrary collection $\Sigma=\{\sigma_1,\ldots,\sigma_n\}$ need of course not be a semigroup, but it does generate a unique smallest semigroup
$$\Sigma'=\{\sigma_1, \ldots, \sigma_n, \sigma_{n+1},\ldots, \sigma_{n'}\} \ \mbox{that contains}\ \Sigma\, .$$
In fact, every coupled cell network map $\gamma_f$ subject to $\Sigma$ is also a coupled cell network map subject to the semigroup $\Sigma'$. Indeed, if we define 
$$f'(X_1, \ldots, X_n, X_{n+1}, \ldots, X_{n'}) := f(X_1, \ldots, X_n)$$ 
then it obviously holds that 
$$(\gamma_{f'})_i(x) = f'(x_{\sigma_1(i)}, \ldots, x_{\sigma_n(i)}, x_{\sigma_{n+1}(i)}, \ldots, x_{\sigma_{n'}(i)}) = f(x_{\sigma_1(i)}, \ldots, x_{\sigma_n(i)}) = (\gamma_{f})_i(x)\ .$$
For this reason, we will throughout this paper always augment $\Sigma$ to the semigroup $\Sigma'$ and think of every coupled cell network map subject to $\Sigma$ as a (special case of a) coupled cell network map subject to the semigroup $\Sigma'$. 

To illustrate that this augmentation is natural, let us finish this section by mentioning a result from \cite{CCN} concerning the synchronous solutions of a network dynamical system. We recall the following well-known definition:

\begin{definition}
Let $\Sigma=\{\sigma_1, \ldots, \sigma_n\}$ be a collection of maps, not necessarily forming a semigroup, and
$P=\{P_1,\ldots, P_r\}$ a partition of $\{1,\dots, N\}$.
If the subspace 
$$ {\rm Syn}_{P} :=\{x\in V^N\, |\ x_{i_1}=x_{i_2} \, \mbox{when} \ i_1 \ \mbox{and}\ i_2\ \mbox{are in the same element of}\ P\, \}$$
is an invariant submanifold for the dynamics of $\gamma_f$ for every $f\in C^{\infty}(V^n,V)$, then we call ${\rm Syn}_{P}$ a {\it (robust) synchrony space} for the network defined by $\Sigma$.
\end{definition}
Interestingly, the synchrony spaces of networks subject to $\Sigma$ are the same as those for networks subject to $\Sigma'$. This means that the extension from $\Sigma$ to $\Sigma'$ does not have any effect on synchrony. Indeed, let us state the following result from \cite{CCN}. The proof, that we do not give here, is easy and uses the concept of a {\it balanced partition} of the cells \cite{curious}, \cite{golstew3}, \cite{torok}, \cite{CCN}, \cite{stewart1}, \cite{pivato}.
\begin{lemma}\label{robustness}
Let $\Sigma=\{\sigma_1, \ldots, \sigma_n\}$ be a collection of maps, not necessarily forming a semigroup, and
$P=\{P_1,\ldots, P_r\}$ a partition of $\{1,\dots, N\}$. Then ${\rm Syn}_P$ is a synchrony space for $\Sigma$ if and only if it is a synchrony space for the semigroup $\Sigma'$ generated by $\Sigma$. 
\end{lemma}
Lemma \ref{robustness} shows that the extension from $\Sigma$ to $\Sigma'$ is harmless from the point of view of synchrony. 

\section{Hidden symmetry}\label{sechidden}
We will now show that every homogeneous coupled cell network is conjugate to a network that is equivariant under a certain action of a semigroup. The symmetry of the latter network thus acts as a hidden symmetry for the original network. 


For the remainder of this paper, let us assume that $\Sigma=\{\sigma_1, \ldots, \sigma_n\}$ is a semigroup (i.e. $\Sigma=\Sigma'$ and the necessary extension has taken place). To understand the hidden symmetries of the networks subject to $\Sigma$, one should note that every $\sigma_j\in \Sigma$ induces a map 
$$\widetilde \sigma_j: \{1,\ldots, n\}\to\{1, \ldots, n\}\ \mbox{via the formula}\ \sigma_{\widetilde \sigma_j(k)} = \sigma_j \circ \sigma_k\, .$$
The map $\widetilde \sigma_j$ encodes the left-multiplicative behavior of $\sigma_j$. We shall write $\widetilde \Sigma := \{\widetilde \sigma_1, \ldots, \widetilde \sigma_n\}$. 
The following result is easy to prove:
\begin{proposition}
If $\Sigma$ is a semigroup with unit, then so is $\widetilde \Sigma$ and the map 
$$\sigma_j \mapsto \widetilde \sigma_j\ \mbox{from}\ \Sigma\ \mbox{to}\ \widetilde \Sigma$$ is a homomorphism of semigroups.
\end{proposition}
\begin{proof}
By definition, it holds for all $i,j,k$ that
$$\sigma_{\widetilde{ (\sigma_i \circ \sigma_j)}(k)} = (\sigma_i \circ \sigma_j) \circ \sigma_k  = \sigma_i\circ (\sigma_j\circ \sigma_k)= \sigma_i\circ \sigma_{\widetilde \sigma_j(k)} = \sigma_{\widetilde \sigma_i ( \widetilde \sigma_j(k))}\, . $$
Because $\Sigma$ is a semigroup (and hence its elements are distinct), this implies that
$$\widetilde{\sigma_{i}\circ\sigma_{j}} =\widetilde \sigma_{i}\circ \widetilde \sigma_{j}\, .$$
In other words, $\widetilde \Sigma$ is closed under composition and the map $\sigma_j \mapsto \widetilde \sigma_j$ is a homomorphism.
  
It remains to check that if $\Sigma$ has a unit, then so does $\widetilde \Sigma$ and its elements are distinct. So let us assume that $\sigma_{i^*}$ is the unit of $\Sigma$, i.e. that $\sigma_{i^*}\circ \sigma_{j} = \sigma_j\circ \sigma_{i^*}= \sigma_j$ for all $j$.
Then
$$\sigma_{\widetilde \sigma_{i^*}(j)} = \sigma_{i^*}\circ \sigma_{j} = \sigma_j\, .$$
This means that $\widetilde \sigma_{i^*}={\rm id}_{\{1, \ldots, n\}}$ and hence that $\widetilde \Sigma$ has a unit.

Finally, it also follows for $j\neq k$ that 
$$\sigma_{\widetilde \sigma_{j}(i^*)} = \sigma_{j}\circ \sigma_{i^*}=\sigma_j \neq \sigma_k = \sigma_{k}\circ \sigma_{i^*}=\sigma_{\widetilde \sigma_{k}(i^*)}$$ and therefore that the elements of $\widetilde \Sigma$ are distinct.
\end{proof}

\begin{definition}
We shall call a semigroup with (left- and right-)unit a {\it monoid}.
\end{definition}

\begin{example}\label{running5}
Recall our running Example \ref{running} in which 
$$\sigma_1[123]:=[123], \sigma_2[123]:=[121]\ \mbox{and} \ \sigma_3[123]:=[111]\, .$$ 
From the composition table of $\Sigma$ given in Example \ref{comptable}, we see that $\widetilde \sigma_1, \widetilde \sigma_2, \widetilde \sigma_3$ are given by
$$\widetilde \sigma_1[123]=[123], \widetilde \sigma_2[123]=[223]\ \mbox{and}\ \widetilde \sigma_3[123]=[333]\, .$$
These maps are not conjugate to the maps $\sigma_1, \sigma_2, \sigma_3$ by any permutation of the cells $\{1,2,3\}$ but do have the same composition table.

A graphical representation of the network maps $\widetilde \sigma_1, \widetilde \sigma_2, \widetilde \sigma_3$ is given in Figure \ref{pict2}.

\begin{figure}[ht]\renewcommand{\figurename}{\rm \bf \footnotesize Figure}
\centering
\begin{tikzpicture}[->,>=stealth',shorten >=1pt,auto,node distance=1.5cm,
                    thick,main node/.style={circle,draw,font=\sffamily\large\bfseries}]

  \node[main node] (1) {$x_1$};
  \node[main node] (2) [below of=1] {$x_2$};
  \node[main node] (3) [below of=2] {$x_3$};

   \node[main node] (4) [right of=1] {$x_1$};
  \node[main node] (5) [below of=4] {$x_2$};
  \node[main node] (6) [below of=5] {$x_3$};
  \node[main node] (7) [right of =4] {$x_1$};
  \node[main node] (8) [below  of=7] {$x_2$};
  \node[main node] (9) [below of=8] {$x_3$};

  \path[every node/.style={font=\sffamily\small}]
(1) edge [loop below] node {} (1)
(2) edge [loop below] node {} (2)
(3) edge [loop below] node {} (3)

(5) edge node {} (4)
(5) edge [loop below] node {} (5)
(6) edge [loop below] node {} (6)

(9) edge node {} (8)
(9) edge [bend right] node {} (7)
(9) edge [loop below] node {} (9)
    ;

\end{tikzpicture}
\caption{\footnotesize {\rm The collection $\{\widetilde \sigma_1, \widetilde \sigma_2, \widetilde \sigma_3\}$ depicted as a directed multigraph.}}
\label{pict2}
\end{figure}
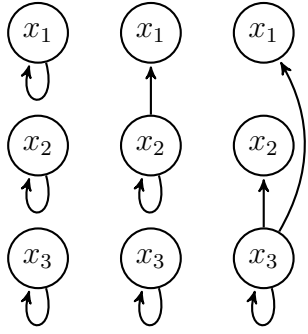

\end{example}
One can of course also study coupled cell networks subject to the monoid $\widetilde \Sigma$. They give rise to a differential equation on $V^n$ of the form 
$$\dot X_j = f(X_{\widetilde \sigma_1(j)}, \ldots, X_{\widetilde \sigma_n(j)})\ \mbox{for}\ 1\leq j \leq n\, .$$
These differential equations will turn out important enough to give the corresponding maps and vector fields a special name:

\begin{definition}
Let $\Sigma=\{\sigma_1, \ldots, \sigma_n\}$ be a monoid and $f:V^n\to V$ a smooth function. Then we call the coupled cell network map/vector field
$$\Gamma_f:V^n\to V^n\ \mbox{defined as} \ (\Gamma_{f})_j(X) := f (X_{\widetilde \sigma_1(j)}, \ldots, X_{\widetilde \sigma_n(j)}) \ \mbox{for}\ 1\leq j\leq n $$
the {\it fundamental network} of $\gamma_f$.
\end{definition}

\begin{example}\label{fundamentalexample}
For our running Example \ref{running} the maps $\widetilde \sigma_1, \widetilde \sigma_2, \widetilde \sigma_3$ were computed in Example \ref{running5}. We read off that 
the equations of motion of the fundamental network are given by
$$\begin{array}{ll} \dot X_1 = & f(X_1, X_2, X_3) \\  \dot X_2 = & f(X_2, X_2, X_3) \\ \dot X_3 = & f(X_3, X_3, X_3)\end{array} \, .$$
\end{example}

\begin{remark}
If $\Sigma$ is a monoid, then so is $\widetilde \Sigma$ and we can observe that
$$\widetilde \sigma_{\widetilde{\widetilde \sigma}_i(j)} = \widetilde \sigma_i \circ \widetilde \sigma_j = \widetilde{\sigma_i \circ  \sigma_j} = \widetilde{ \sigma_{\widetilde \sigma_i(j)}} =  \widetilde \sigma_{\widetilde \sigma_i(j)}\, . $$
This proves that $\widetilde{\widetilde \sigma}_i = \widetilde \sigma_i$ and thus that $\widetilde{\widetilde \Sigma} = \widetilde \Sigma$. In particular, $\Gamma_f$ is equal to its own fundamental network. In fact, this is the reason we call $\Gamma_f$ ``fundamental''.
\end{remark}
Theorem \ref{conjugation} below was proved in \cite{CCN} and demonstrates the relation between $\gamma_f$ and $\Gamma_f$:

\begin{theorem}\label{conjugation}
For $1\leq i\leq N$ let us define the map $\pi_i:V^N\to V^n$ by 
$$\pi_i(x_1, \ldots, x_N):=(x_{\sigma_1(i)}, \ldots, x_{\sigma_n(i)})\, .$$
All the maps $\pi_i$ conjugate $\gamma_f$ to $\Gamma_f$, that is 
$$\Gamma_f\circ \pi_i = \pi_i\circ \gamma_f\  \ \mbox{for all} \ 1\leq i\leq N \, .$$
\end{theorem}
\begin{proof}
We remark that the definition of $\pi_i:V^N\to V^n$ is such that $$(\gamma_f)_i=f\circ \pi_i\, .$$ 
With this in mind, let us also define for $1\leq j \leq n$ the maps 
\begin{align}\label{Amapdef}
A_{\sigma_j}: V^n\to V^n \ \mbox{by}\ 
 A_{\sigma_j}(X_1, \ldots, X_{n}) :=(X_{\widetilde \sigma_1(j)}, \ldots, X_{\widetilde \sigma_n(j)})
\end{align}
for which it holds that $$(\Gamma_f)_j= f\circ A_{\sigma_j}\, .$$ 
With these definitions, we find that
\begin{align}
(A_{\sigma_j}\circ \pi_i)(x)& =A_{\sigma_j}(x_{\sigma_1(i)}, \ldots, x_{\sigma_n(i)}) = (x_{\sigma_{\widetilde\sigma_1(j)}(i)}, \ldots, x_{\sigma_{\widetilde\sigma_n(j)}(i)}) \nonumber \\ \nonumber  
& = (x_{\sigma_1(\sigma_j(i))}, \ldots, x_{\sigma_n(\sigma_j(i))}) = \pi_{\sigma_j(i)}(x)\, .
\end{align}
In other words, $$A_{\sigma_j}\circ \pi_i =\pi_{\sigma_j(i)}\, .$$ 
As a consequence, we have for $x\in V^N$ that
$$(\Gamma_f\circ\pi_i)_j(x)=f(A_{\sigma_j}(\pi_i(x))) =f(\pi_{\sigma_j(i)}(x)) = (\gamma_f(x))_{\sigma_j(i)} = (\pi_i\circ\gamma_f)_j(x)\, .$$
This proves the theorem.
\end{proof}
Theorem \ref{conjugation} says that $\Gamma_f$ is semi-conjugate to $\gamma_f$. In particular, every $\pi_i$ sends integral curves of $\gamma_f$ to integral curves of $\Gamma_f$ (and discrete-time orbits of $\gamma_f$ to those of $\Gamma_f$). The opposite need not be true though, because it may happen that none of the $\pi_i$ is invertible. 

In addition, the dynamics of $\gamma_f$ can be reconstructed from the dynamics of $\Gamma_f$. More precisely, when $x(t)$ is an integral curve of $\gamma_f$ and $X_{(i)}(t)$ are integral curves of $\Gamma_f$ with $X_{(i)}(0)=\pi_i(x(0))$, then $\pi_i(x(t)) = X_{(i)}(t)$
and thus $$\dot x_i(t) = f(\pi_i(x(t))) = f(X_{(i)}(t)) \ \mbox{for}\ 1\leq i\leq N\, .$$ 
This means that $x(t)$ can simply be obtained by integration. In other words, it suffices to study the dynamics of $\Gamma_f$ to understand the dynamics of $\gamma_f$.  

\begin{example}\label{conjugateexample}
In our running Example \ref{running}, the maps $\pi_1, \pi_2, \pi_3: V^3\to V^3$ are given by 
$$\begin{array}{l} \pi_1(x_1, x_2, x_3) = (x_1, x_1, x_1)\, , \\ \pi_2(x_1, x_2, x_3) = (x_2, x_2, x_1)\, ,\\ \pi_3(x_1, x_2, x_3) = (x_3, x_1, x_1)\, . \end{array}$$
Although none of these maps is invertible, they indeed send the solutions of 
$$\begin{array}{ll} \dot x_1 = & f(x_1, x_1, x_1) \\  \dot x_2 = & f(x_2, x_2, x_1) \\ \dot x_3 = & f(x_3, x_1, x_1)\end{array} \, $$
to solutions of the equations
$$\begin{array}{ll} \dot X_1 = & f(X_1, X_2, X_3) \\  \dot X_2 = & f(X_2, X_2, X_3) \\ \dot X_3 = & f(X_3, X_3, X_3)\end{array} \, .$$
\end{example}

\begin{remark} One could think of the maps $\pi_i$ as forming ``shadows'' of the dynamics of $\gamma_f$ in the dynamics of $\Gamma_f$, in such a way that the original dynamics can be reproduced from all its shadows.

At the same time, the transition from $\gamma_f$ to $\Gamma_f$ is  reminiscent of the symmetry reduction of an equivariant dynamical system: the dynamics of $\gamma_f$ descends to the dynamics of $\Gamma_f$ and the dynamics of $\gamma_f$ can be reconstructed from that of $\Gamma_f$ by means of integration. 

Most importantly, $\Gamma_f$ captures all the dynamics of $\gamma_f$.
\end{remark}

\begin{remark}\label{equilibriaremark}
When $\gamma_f(x)=0$ then $\Gamma_f(\pi_i(x)) = \pi_i(\gamma_f(x))=0$, that is $\pi_i$ sends equilibria of $\gamma_f$ to equilibria of $\Gamma_f$. On the other hand, when $\Gamma_f(\pi_i(x))=0$, then 
$$(\gamma_f)_{\sigma_{j}(i)}(x)=f(\pi_{\sigma_j(i)}(x)) = f(A_{\sigma_{j}}(\pi_i(x))) = (\Gamma_f)_j(\pi_i(x)) = 0\, .$$ 
Applied to $j=i^*$ this gives that $(\gamma_f)_{i}(x)=0$ if $\Gamma_f(\pi_i(x))=0$, that is $x$ is an equilibrium of $\gamma_f$ as soon as all maps $\pi_i$ send it to an equilibrium.

We conclude that $x\in V^N$ is an equilibrium point of $\gamma_f$ if and only if all the points $\pi_i(x)\in V^n \ (1\leq i\leq N)$ are equilibria of $\Gamma_f$. With this in mind, we can determine the equilibria of $\gamma_f$ from those of $\Gamma_f$.
\end{remark}

\noindent There are two major advantages of studying $\Gamma_f$ instead of $\gamma_f$:
\begin{itemize}
\item[{\bf 1.}] The network structure of $\Gamma_f$ only depends on the composition/multiplication table (i.e. the semigroup/monoid structure) of $\Sigma$. More precisely: when $\Sigma^{(1)}$ and $\Sigma^{(2)}$ are isomorphic monoids, then the fundamental networks $\Gamma_f^{(1)}$ and $\Gamma_f^{(2)}$ are (bi-)conjugate. This is not true in general for $\gamma_f^{(1)}$ and $\gamma_f^{(2)}$. Thus, every abstract monoid corresponds to precisely one fundamental network.
\item[{\bf 2.}] The fundamental network $\Gamma_f$ is fully characterized by symmetry. This is the content of Theorem \ref{equithm} below, and one of the crucial points of this paper.
\end{itemize}


\begin{theorem}\label{equithm}
Let $\Sigma=\{\sigma_1, \ldots, \sigma_n\}$ be a monoid with unit $\sigma_{i^*}$ and define the maps
\begin{align}\label{Amapdef2}
A_{\sigma_j}: V^n\to V^n \ \mbox{by}\ 
 A_{\sigma_j}(X_1, \ldots, X_{n}) :=(X_{\widetilde \sigma_1(j)}, \ldots, X_{\widetilde \sigma_n(j)})\ \mbox{for}\ 1\leq j\leq n\, .
\end{align}
Then the following are true:
\begin{itemize}
\item The maps $A_{\sigma_j}$ form a representation of $\Sigma$ in $V^n$, i.e. $A_{\sigma_{i^*}}={\rm id}_{V^n}$ and
$$A_{\sigma_j}\circ A_{\sigma_k} = A_{\sigma_j\circ \sigma_k}\ \mbox{for all}\ 1\leq j, k \leq n\, .$$
\item Each fundamental network $\Gamma_f:V^n\to V^n$ is equivariant under this representation:
\begin{align}\label{equi}
\Gamma_f\circ A_{\sigma_j}=A_{\sigma_j}\circ \Gamma_f \ \mbox{for all}\ 1\leq j \leq n\, .
\end{align}
\item Conversely, if $\Gamma:V^n\to V^n$ satisfies $$\Gamma\circ A_{\sigma_j}=A_{\sigma_j}\circ \Gamma\ \mbox{for all}\ 1\leq j \leq n\, ,$$ 
then there exists a function $f:V^n\to V$ such that 
$$\Gamma_j(X)= (\Gamma_f)_j(X) = f(A_{\sigma_j}X)\ \mbox{for all}\ 1\leq j\leq n\, .$$ 
\end{itemize}
\end{theorem}

\begin{proof}
Because $\sigma_{\widetilde \sigma_j(i^*)} = \sigma_j\circ \sigma_{i^*} = \sigma_j$, it holds that $\widetilde \sigma_j(i^*) = j$ and hence that $A_{\sigma_{i^*}}(X) = (X_{\widetilde \sigma_1(i^*)}, \ldots, X_{\widetilde \sigma_n(i^*)})=X$. This proves that $A_{\sigma_{i^*}}={\rm id}_{V^n}$. Furthermore, by definition
$$\sigma_{\widetilde \sigma_{\widetilde \sigma_{k}(j_1)}(j_2)} = \sigma_{\widetilde \sigma_{k}(j_1)}\circ \sigma_{j_2}= \sigma_k\circ\sigma_{j_1}\circ\sigma_{j_2} =  \sigma_k\circ\sigma_{\widetilde \sigma_{j_1}(j_2)}= \sigma_{\widetilde \sigma_k(\widetilde \sigma_{j_1}(j_2))} \, .$$
This implies that $\widetilde \sigma_{\widetilde \sigma_{k}(j_1)}(j_2) = \widetilde \sigma_k(\widetilde \sigma_{j_1}(j_2))$, which in turn yields that
\begin{align}\nonumber
(&A_{\sigma_{j_1}}\circ A_{\sigma_{j_2}}) (X_1, \ldots, X_n) = A_{\sigma_{j_1}}(X_{\widetilde \sigma_{1}(j_2)}, \ldots, X_{\widetilde \sigma_{n}(j_2)})=  (X_{\widetilde \sigma_{\widetilde \sigma_{1}(j_1)}(j_2)}, \ldots, X_{\widetilde \sigma_{\widetilde \sigma_{n}(j_1)}(j_2)})   \\ \nonumber 
&= (X_{\widetilde \sigma_{1}(\widetilde \sigma_{j_1}(j_2))}, \ldots, X_{\widetilde \sigma_{n}(\widetilde \sigma_{j_1}(j_2))}) = A_{\sigma_{\widetilde \sigma_{j_1}(j_2)}}(X_1, \ldots, X_n)= A_{\sigma_{j_1}\circ \sigma_{j_2}}(X_1, \ldots, X_n) \ .
\end{align}
This proves the first claim of the theorem. 

The second claim from the first claim and from the fact that $(\Gamma_f)_k=f\circ A_{\sigma_k}$:
\begin{align}
(\Gamma_f\circ A_{\sigma_j})_k(X) &= f(A_{\sigma_k}\circ A_{\sigma_j}X) = f(A_{\sigma_k\circ\sigma_j}X)= \nonumber \\ \nonumber f(A_{\sigma_{\widetilde \sigma_k(j)}}X) &= (\Gamma_f)_{\widetilde \sigma_k(j)}(X) = (A_{\sigma_j}\circ \Gamma_f)_k(X)\, .
\end{align}
To prove the third claim, assume that
$A_{\sigma_j}\circ \Gamma = \Gamma \circ A_{\sigma_j}$ for all $1\leq j\leq n$. Recalling that $(A_{\sigma_j}\Gamma)_i = \Gamma_{\widetilde \sigma_i(j)}$, this  implies that 
\begin{align}\label{symmetry}
\Gamma_{\widetilde \sigma_i(j)}(X) = (A_{\sigma_j}\circ \Gamma)_i(X) =
(\Gamma \circ A_{\sigma_j})_i(X)  = \Gamma_i(A_{\sigma_j}X)\, .
\end{align}
When $\sigma_{i^*}$ is the unit of $\Sigma$, then $\sigma_{\widetilde \sigma_{i^*}(j)}=\sigma_{i^*}\circ \sigma_{j}=\sigma_j$, so $\widetilde \sigma_{i^*}={\rm id}_{\{1,\ldots, n\}}$. As a consequence, applied to $i=i^*$, equation (\ref{symmetry}) implies in particular that 
$$\Gamma_{j}(X) = \Gamma_{i^*}(A_{\sigma_j}X)\ \mbox{for all}\ j=1, \ldots, n\, .$$
If we now choose $f:=\Gamma_{i^*}, V^n\to V$, then $\Gamma=\Gamma_f$ as required. 
\end{proof}
\noindent Theorem \ref{equithm} says that a vector field $\Gamma: V^n\to V^n$ is a fundamental coupled cell network for the monoid $\Sigma$ if and only if it is equivariant under the action $\sigma_j\mapsto A_{\sigma_j}$ of this monoid. In particular, all the degeneracies that occur in the dynamics of the fundamental network $\dot X = \Gamma_f(X)$ are due to (monoid-)symmetry. Such degeneracies may include the existence of synchrony spaces and the occurrence of double eigenvalues.

\begin{remark}
We shall be referring to the transformations $A_{\sigma_j}$ as ``symmetries'' of $\Gamma_f$, even if these transformations may not be invertible.  

As is the case for groups of (invertible) symmetries, our semigroup of symmetries can force the existence of synchronous solutions. For instance, the fixed point set of any of the maps $A_{\sigma_j}$ is flow-invariant for any $\Sigma$-equivariant vector field. Because the $A_{\sigma_j}$ may not be invertible, there can be many more invariant subsets though. For example, the image ${\rm im}\, A_{\sigma_j}$ of a symmetry is flow-invariant and the inverse image $A_{\sigma_j}^{-1}(W)$ of a flow-invariant subspace $W$ is flow-invariant. We conclude that synchrony spaces can arise as symmetry spaces in many different ways.

Also, recall from Theorem \ref{conjugation} that the maps $\pi_i: V^N\to V^n$ send orbits of $\gamma_f$ to orbits of $\Gamma_f$. As a consequence, ${\rm im}\, \pi_i \subset V^n$ is invariant under the flow of $\Gamma_f$. In fact, it was proved in \cite{CCN} that this image is a robust synchrony space corresponding to a balanced partition of the cells of the fundamental network.
\end{remark}

\begin{example}
The fundamental network of our running Example \ref{running} was given by 
\begin{align}\label{fundamentalexampleformula}
\begin{array}{ll} \dot X_1 = & f(X_1, X_2, X_3) \\  \dot X_2 = & f(X_2, X_2, X_3) \\ \dot X_3 = & f(X_3, X_3, X_3)\end{array} \, .
\end{align}
In other words, the representation of the monoid $\{\sigma_1, \sigma_2, \sigma_3\}$ is given by
\begin{align}\nonumber
A_{\sigma_1}(X_1, X_2, X_3) = (X_1, X_2, X_3)\, ,  \\ \nonumber A_{\sigma_2}(X_1, X_2, X_3) = (X_2, X_2, X_3)\, , \\ \nonumber A_{\sigma_3}(X_1, X_2, X_3) = (X_3, X_3, X_3)\, .
\end{align}
One checks that this representation indeed consists of symmetries of (\ref{fundamentalexampleformula}). In this example, the nontrivial balanced partitions of the fundamental network are 
$$\{1,2,3\},  \{1,2\} \cup \{3\} \ \mbox{and}\ \{1\}\cup\{2,3\} \, .$$
These respectively correspond to the robust synchrony spaces 
$$\{X_1=X_2=X_3\}, \{X_1=X_2\} \ \mbox{and}\ \{X_2=X_3\}\, .$$
These synchrony spaces can both be characterized in terms of the conjucacies $\pi_i$ and in terms of the symmetries $A_{\sigma_j}$, namely
\begin{align}\nonumber
& \{X_1=X_2=X_3\}={\rm im}\, \pi_1 = {\rm Fix}\, A_{\sigma_3}= {\rm im}\, A_{\sigma_3}\, , \\ \nonumber
& \{X_1=X_2\}={\rm im}\, \pi_2 = {\rm Fix}\, A_{\sigma_2} = {\rm im} \, A_{\sigma_2} \ \mbox{and}\\ \nonumber 
& \{X_2=X_3\}={\rm im}\, \pi_3 = A_{\sigma_2}^{-1}({\rm im}\, A_{\sigma_3}) \, .
\end{align}
\end{example}


\section{Representations of semigroups}\label{secrepresentations}
In this section we present some rather well-known facts from the representation theory of semigroups. We choose to explain and prove these results in great detail, as our readers  may not be so familiar with them. One of the main goals of this section is to explain how semigroup symmetry can lead to degeneracies in the spectrum of the linearization of an equivariant vector field at a symmetric equilibrium. Although the theory in this section has strong similarities with the representation theory of compact groups, we would like to warn the reader in advance that the situation is slightly more delicate for semigroups. 

 Firstly, when $\Sigma$ is a semigroup and $W$ a finite dimensional real vector space, then a map 
$$A: \Sigma \to \mathfrak{gl}(W)\ \mbox{for which}\ A_{\sigma_j}\circ A_{\sigma_k} = A_{\sigma_j\circ \sigma_k}\ \mbox{for all} \ \sigma_j, \sigma_k\in\Sigma$$ 
will be called a representation of the semigroup $\Sigma$ in $W$.

A subspace $W_1\subset W$ is called a subrepresentation of $W$ if it is stable under the action the semigroup, that is if $A_{\sigma_j}(W_1)\subset W_1$ for all $\sigma_j\in \Sigma$.
\begin{definition}
A representation of $\Sigma$ in $W$ is called {\it indecomposable} if $W$ is not a direct sum $W=W_1\oplus W_2$ with $W_1$ and $W_2$ both nonzero subrepresentations of $W$. 

A representation of $\Sigma$ in $W$ is called {\it irreducible} if $W$ does not contain a subrepresentation $W_1\subset W$ with $W_1 \neq 0$ and $W_1\neq W$. 
\end{definition}
Clearly, an irreducible representation is indecomposable, but the converse is not true. By definition, every representation is a direct sum of indecomposable representations. In general, this statement is not true for irreducible representations either (unless for example $\Sigma$ is a compact group). Also, the decomposition of a representation into indecomposable representations is unique up to isomorphism. This is the content of the Krull-Schmidt theorem that we will formulate and prove below. 

When $A:\Sigma \to \mathfrak{gl}(W)$ and $A':\Sigma\to \mathfrak{gl}(W')$ are semigroup representations and $L: W\to W'$ is a linear map so that 
$$L\circ A_{\sigma_j} = A'_{\sigma_j}\circ L \ \mbox{for all}\ \sigma_j\in \Sigma \, ,$$
then we call $L$ a {\it homomorphism of representations} and write $L\in {\rm Hom}(W, W')$ - note that the dependence on the semigroup $\Sigma$ is not expressed by our notation. We shall call an invertible homomorphism an {\it isomorphism}. The following result will be useful later:
\begin{proposition}\label{twoisos}
Let $W, W'$ be indecomposable representations, $L_1\in {\rm Hom}(W, W')$ and $L_2\in {\rm Hom}(W', W)$. 
If $L_2\circ L_1$ is invertible, then both $L_1$ and $L_2$ are isomorphisms.
\end{proposition}
\begin{proof}
Note first of all that when $L_1\in {\rm Hom}(W, W')$ and $Y=L_1(X)$, then $A_{\sigma_j}'(Y) = A_{\sigma_j}'(L_1(X))=L_1(A_{\sigma_j}(X))$ and thus ${\rm im}\, L_1$ is a subrepresentation of $W'$. Similarly, when $L_2\in {\rm Hom}(W', W)$ and $L_2(X)=0$, then $L_2(A'_{\sigma_j}(X)) = A_{\sigma_j}(L_2(X))= 0$ and hence also $\ker L_2$ is a subrepresentation of $W'$. 

Because $L_2\circ L_1$ is invertible, it must hold that $L_1$ is injective, $L_2$ is surjective and $ {\rm im}\, L_1\cap \ker L_2=\{0\}$. Now, let $Z\in W'$. Then $Z= X+Y$ with $Y=L_1\circ (L_2\circ L_1)^{-1}\circ L_2 (Z)$ and $X=Z-Y$. It is clear that $X \in \ker L_2$ and that $Y \in {\rm im}\, L_1$  and we conclude that $W'={\rm im}\, L_1 \oplus \ker L_2$. But both ${\rm im}\, L_1$ and $\ker L_2$ are subrepresentations of $W'$ and $W'$ was assumed indecomposable. We conclude that $\ker L_2=0$ and ${\rm im} \, L_1 = W'$ and hence that $L_1$ is surjective and $L_2$ is injective. Thus, $L_1$ and $L_2$ are isomorphisms.
\end{proof} 
If $W$ is a semigroup representation, then we call an element of ${\rm Hom}(W,W)$ an {\it endomorphism} of $W$ and we shall write ${\rm End}(W):={\rm Hom}(W,W)$. 

\begin{remark}\label{linearization}
Assume that $\Gamma: W\to W$ is a vector field and $X_0\in W$ is a point such that
\begin{itemize}
\item[i)] $X_0$ is an equilibrium point of $\Gamma$, i.e. $\Gamma(X_0)=0$,
\item[ii)] $X_0$ is $\Sigma$-symmetric, i.e. $A_{\sigma_j}(X_0)=X_0$ for all $1\leq j \leq n$ and
\item[iii)] $\Gamma$ is $\Sigma$-equivariant, i.e. $\Gamma \circ A_{\sigma_j}=A_{\sigma_j} \circ \Gamma$ for all $1\leq j\leq n$.
\end{itemize}
Then differentiation of $\Gamma(A_{\sigma_j}(X))=A_{\sigma_j}(\Gamma(X))$ at $X=X_0=A_{\sigma_j}(X_0)$ yields that
$$L_0\circ A_{\sigma_j} = A_{\sigma_j}\circ L_0\ \mbox{if we set} \ L_0:= D_X\Gamma(X_0)\, .$$
This explains why we are interested in the endomorphisms of a representation of a semigroup: the linearization of an equivariant vector field at a symmetric equilibrium is an example of such an endomorphism.
\end{remark}
The following proposition states that the endomorphisms of an indecomposable representation fall into two classes.

\begin{proposition}\label{oneev}
Let $W$ be an indecomposable representation and $L\in {\rm End}(W)$. 
Then $L$ is either invertible or nilpotent.
\end{proposition}
\begin{proof}
Let $n$ be large enough that $W=\ker L^n \oplus {\rm im}\, L^n$. Because $\ker L^n$ and ${\rm im}\, L^n$ are subrepresentations of $W$ and $W$ is indecomposable it follows that either $W=\ker L^n$ and $L$ is nilpotent, or $W={\rm im}\, L^n$ and $L$ is invertible. 
\end{proof} 
As a consequence of Proposition \ref{oneev}, we find that the endomorphisms of indecomposable representations have spectral degeneracies as follows.

\begin{corollary}
Let $W$ be an indecomposable representation and $L\in {\rm End}(W)$. Then either
\begin{itemize}
\item[i)] $L$ has one real eigenvalue, or
\item[ii)]  $L$ has one pair of complex conjugate eigenvalues.
\end{itemize}
\end{corollary}
\begin{proof}
First, assume that $\lambda$ is a real eigenvalue of $L$. Then $L-\lambda I$ is not invertible and hence nilpotent, i.e. $(L-\lambda I)^n=0$ for some $n$. It follows that every element of $W$ is a generalized eigenvector for the eigenvalue $\lambda$.

The argument is similar in case that $\lambda=\alpha\pm i\beta\notin \R$ is a complex conjugate pair of eigenvalues of $L$. Then $(L-\lambda I)(L-\overline \lambda I) = L^2-2\alpha L +(\alpha^2+\beta^2)I$ is not invertible and thus nilpotent.
\end{proof}
Schur's lemma gives a more precise characterization of ${\rm End}(W)$. We will formulate this characterization as Lemma \ref{schur} below. It follows from a few preparatory results. The first says that the endomorphisms of an indecomposable representation form a ``local ring''.
\begin{proposition}\label{sumnilpotent}
Let $W$ be an indecomposable representation and assume that $L_1, L_2\in {\rm End}(W)$ are both nilpotent. Then also $L_1+L_2$ is nilpotent.
\end{proposition}
\begin{proof}
Assume that $L_1+L_2$ is not nilpotent. Then it is invertible by Proposition \ref{oneev}. Multiplying 
$$(L_1+L_2)\circ A_{\sigma_j} = A_{\sigma_j}\circ (L_1+L_2)$$ left and right by $(L_1+L_2)^{-1}$ gives that $(L_1+L_2)^{-1}\in {\rm End}(W)$. Hence, so are $$M_1:= (L_1+L_2)^{-1}L_1\ \mbox{and} \ M_2:=(L_1+L_2)^{-1}L_2\, .$$
Clearly, $M_1$ and $M_2$ can not be invertible, because they contain a nilpotent factor. So they are both nilpotent. In particular, $I-M_2$ is invertible. But $M_1=I-M_2$. This is a contradiction.
\end{proof}
\begin{corollary}\label{nilpotentideal}
Let $W$ be an indecomposable representation. Then the collection $${\rm End}^{\rm Nil}(W)=\{L\in {\rm End}(W)\, |\, L \ \mbox{\rm is nilpotent}\, \! \}$$ is an ideal in ${\rm End}(W)$.
\end{corollary}
\begin{proof}
Obvious from Proposition \ref{oneev} and Proposition \ref{sumnilpotent}.
\end{proof}
We can now formulate the following refinement of Proposition \ref{oneev}:
\begin{lemma}[Schur's Lemma]\label{schur}
Let $W$ be an indecomposable representation. The quotient
$${\rm End}(W)/{\rm End}^{\rm Nil}(W) \ \mbox{is a division algebra}\, .$$
\end{lemma}
\begin{proof}
By Corollary \ref{nilpotentideal}, the quotient ring ${\rm End}(W)/{\rm End}^{\rm Nil}(W)$ is well-defined. By Proposition \ref{oneev}, an element of this quotient is invertible if and only if it is nonzero. Thus, the quotient is a division algebra.
\end{proof}
We recall that any finite dimensional real associative division algebra is isomorphic to either
$$\R, \C\ \mbox{or}\ \H\, .$$
In particular, this implies that ${\rm End}(W)/{\rm End}^{\rm Nil}(W)$ can only have dimension $1$, $2$ or $4$ if $W$ is indecomposable. We also note that one can represent the equivalence class $[L]\in {\rm End}(W)/{\rm End}^{\rm Nil}(W)$ of an endomorphism $L\in {\rm End}(W)$ by the semisimple part $L^S$ of $L$.  It holds that $L^S\in {\rm End}(W)$, because $L^S$ is a polynomial expression in $L$, and thus Lemma \ref{schur} can also be thought of as a restriction on the semisimple parts of the endomorphisms of an indecomposable representation.

In case ${\rm End}(W)/{\rm End}^{\rm Nil}(W)\cong \R$, we call $W$ a representation of {\it real type} or an {\it absolutely indecomposable} representation. This terminology is due to the fact that even the complexification of $W$ can not be decomposed further. Otherwise we call $W$ an {\it absolutely decomposable} representation, respectively of {\it complex type} or of {\it quaternionic type}.

Finally, let us for completeness prove the Krull-Schmidt theorem here.
\begin{theorem}[Krull-Schmidt] The decomposition
$$W=W_1\oplus \ldots \oplus W_m$$ 
of $W$ into indecomposable representations is unique up to isomorphism.
\end{theorem} 
\begin{proof}
We will prove a slightly more general fact. Let us assume that $W$ and $W'$ are two isomorphic representations of $\Sigma$. This means that there is an isomorphism of representations $h:W\to W'$. Let us assume furthermore that 
$$W=W_1\oplus \ldots \oplus W_m \ \mbox{and}\ W'= W_1'\oplus \ldots \oplus W_{m'}'$$
are decompositions of $W$ and $W'$ into indecomposable representations. We claim that $m=m'$ and that it holds after renumbering the factors  that $W_j$ is isomorphic to $W_j'$ for every $1\leq j\leq m$. The theorem then follows by choosing $W=W'$.

To prove our claim, let 
$$i_{j}:W_j\to W, p_j:W\to W_j, i_{k}':W_k'\to W'\ \mbox{and}\ p_k':W'\to W_k'$$ 
be the embeddings and projections associated to the above decompositions. These maps are homomorphisms of the corresponding representations. 
Indeed, because $A_{\sigma_i}(W_l)\subset W_l$,
\begin{align}\nonumber
&p_{j}\circ A_{\sigma_i} = p_j\circ A_{\sigma_i} \circ( i_1\circ p_1+\ldots +i_n\circ p_m ) = (p_j\circ A_{\sigma_i}\circ i_j) \circ p_j \, , \\ \nonumber
& A_{\sigma_i}\circ i_j = (i_1\circ p_1+\ldots +i_n\circ p_m) \circ (A_{\sigma_i}\circ i_j) =  i_j\circ (p_j \circ A_{\sigma_i}\circ i_j)\, ,
\end{align}
and similarly for $i_k'$ and $p_k'$.

It now holds for $j=1, \ldots, m$ that 
$$\sum_{k=1}^{m'} (p_j \circ h^{-1} \circ i_k' \circ p_k' \circ h \circ i_j) = p_j \circ h^{-1} \left(\sum_{k=1}^{m'}  i_k' \circ p_k' \right)\circ h \circ i_j = p_j\circ i_j= I_{W_j} \, .$$
Because $I_{W_j}$ is invertible, Proposition \ref{sumnilpotent} implies that there is at least one $1\leq k\leq m'$ so that $p_j \circ h^{-1}\circ i_k' \circ p_k'\circ h \circ i_j$ is an isomorphism. The latter is the composition of the homomorphisms
$$p_k'\circ h \circ i_j:W_j\to W_{k}'\ \mbox{and} \  p_j \circ h^{-1}\circ i_k': W_k'\to W_j \, .$$ 
Because $W_j$ and $W_k'$ are indecomposable, it follows from Proposition \ref{twoisos} above that both these maps are isomorphisms, i.e. $W_j$ is isomorphic to $W_k'$. 
The theorem now follows easily by induction. 
\end{proof}

\section{Equivariant Lyapunov-Schmidt reduction}\label{seclyapunovschmidt}
We say that a parameter dependent differential equation 
$$\dot X = \Gamma(X;\lambda)\ \mbox{for}\ X\in W\ \mbox{and}\ \lambda\in \R^p$$
undergoes a steady state bifurcation at $(X_0; \lambda_0)$ if $\Gamma(X_0; \lambda_0)=0$ and the linearization
$$L_0:=D_X\Gamma(X_0;\lambda_0): W\to W$$ 
is not invertible. When this happens, the collection of steady states of $\Gamma$ close to $X_0$ may topologically change as $\lambda$ varies near $\lambda_0$. To study such a bifurcation in detail, it is customary to use the method of Lyapunov-Schmidt reduction, cf. \cite{duisbif}, \cite{constrainedLS}, \cite{golschaef1}, \cite{golschaef2}. We will now describe a variant of this method that applies in case that $\Gamma$ is equivariant under the action of a semigroup. This section serves as a preparation for Sections \ref{secgeneric} and \ref{sectwoorthree} below.

As a start, let us denote by $$L_0=L^S_0+L^N_0$$ the decomposition of $L_0$ in semisimple and nilpotent part. We shall split $W$ as a direct sum
$$W= {\rm im}\, L^S_0 \oplus \ker L^S_0\, .$$
The projections that correspond to this splitting shall be denoted
$$P_{\rm im}:W\to {\rm im}\, L_0^S\ \mbox{and}\ P_{\ker}: W\to \ker L_0^S\, .$$
One can now decompose every element $X\in W$ as $X=X_{\rm im}+X_{\ker}$ with $X_{\rm im}=P_{\rm im}(X)\in {\rm im}\, L_0^S$ and $X_{\rm ker}=P_{\ker}(X)\in \ker L_0^S$. We also observe that $\Gamma(X;\lambda)=0$ if and only if 
$$
\Gamma_{\rm im}(X;\lambda):=P_{\rm im}(\Gamma(X;\lambda)) =0 \ \mbox{and}\ \Gamma_{\rm ker}(X;\lambda):=P_{\rm ker}(\Gamma(X;\lambda)) =0\, .$$
The idea of Lyapunov-Schmidt reduction is to solve these equations consecutively. Thus, one first considers the equation
$$
\Gamma_{\rm im}(X_{\rm im}+X_{\ker};\lambda) =0 \ \mbox{for}\ \Gamma_{\rm im}:{\rm im}\, L_0^S \oplus \ker L_0^S \times \R^p \to {\rm im}\, L_0^S\, .$$ 
By construction, the derivative of $\Gamma_{\rm im}$ in the direction of ${\rm im}\, L_0^S$ is given by
$$D_{X_{\rm im}}\Gamma_{\rm im}(X_0;\lambda_0) = P_{\rm im}\circ D_{X_{\rm im}}\Gamma(X_0;\lambda_0) = P_{\rm im} \circ \left(\left. L_0\right|_{{\rm im}\, L_0^S}\right)\, .$$
This derivative is clearly invertible. As a consequence, by the implicit function theorem there exists a smooth function $X_{\rm im}=X_{\rm im}(X_{\ker}, \lambda)$, defined for $X_{\ker}$ near $(X_0)_{\ker}$ and $\lambda$ near $\lambda_0$ so that $X_{\rm im}(X_{\ker}, \lambda)$ is the unique solution near $(X_{0})_{\rm im}$ of the equation
$$\Gamma_{\rm im}(X_{\rm im}(X_{\ker}, \lambda)+X_{\ker}; \lambda)=0\, .$$
Hence, what remains is to solve the bifurcation equation 
$$r(X_{\ker}; \lambda):= \Gamma_{\ker}(X_{\rm im}(X_{\ker}, \lambda)+X_{\ker}; \lambda)=0\ \mbox{with}\ r: \ker L_0^S\times\Lambda \to \ker L_0^S$$
and $\Lambda\subset  \R^p$ an open neighborhood of $\lambda_0$.

The process described here, of reducing the equation $\Gamma(X;\lambda)=0$ for $X\in W$ to the lower-dimensional equation $r(X_{\ker}; \lambda)=0$ for $X_{\ker}\in \ker L_0^S$ is called {\it Lyapunov-Schmidt reduction}. The following lemma contains the simple but important observation that at a symmetric equilibrium, the reduced equation inherits the symmetries of the original equation. The proof is standard:
\begin{lemma}[Equivariant Lyapunov-Schmidt reduction]
Assume that $W$ is a representation of a semigroup $\Sigma$ and that $\Gamma: W\times \R^p \to W$ is $\Sigma$-equivariant, i.e. that
$$\Gamma(A_{\sigma_j}(X);\lambda)=A_{\sigma_j}(\Gamma(X;\lambda))\ \mbox{for all}\ \sigma_j\in \Sigma\, .$$
Assume furthermore that $X_0\in W$ is a symmetric equilibrium at the parameter value $\lambda_0$, i.e. that $\Gamma(X_0; \lambda_0)=0$ and $A_{\sigma_j}(X_0)=X_0$ for all $1\leq j\leq n$. 

Then also the reduced vector field $r: \ker L_0^S\times \Lambda \to \ker L_0^S$ is $\Sigma$-equivariant:
$$r(A_{\sigma_j}(X_{\ker}); \lambda) = A_{\sigma_j}(r(X_{\ker}; \lambda))\ \mbox{for all}\ \sigma_j\in \Sigma\, ,$$
where we denoted by $A_{\sigma_j}: \ker L_0^S\to \ker L_0^S$ the restriction of $A_{\sigma_j}$ to $\ker L_0^S$.
\end{lemma}
\begin{proof}
First of all,  because $\Gamma$ is $\Sigma$-equivariant and $X_0$ is $\Sigma$-symmetric, we know from Remark \ref{linearization} that
$$L_0\in {\rm End}(W)\, .$$ Moreover, $L^S_0\in {\rm End}(W)$ as well, because $L_0^S=p(L_0)=a_0I+a_1L+\ldots+a_{r-1}L^{r-1}+a_r L^r$ for certain $a_0, \ldots, a_r\in \C$. As a consequence, $\ker L_0^S$ and ${\rm im}\, L_0^S$ are subrepresentations of $W$ and the projections $P_{\ker}: W\to \ker L_0^S$ and $P_{\rm im}: W\to{\rm im}\, L_0^S$ are homomorphisms.



Recall that $X_{\rm im}(X_{\ker}, \lambda)$ is the unique solution to the equation $\Gamma_{\rm im}(X_{\rm im}+X_{\ker}; \lambda)=0$. In particular it holds that $\Gamma_{\rm im}(X_{\rm im}(A_{\sigma_j}(X_{\ker}), \lambda) + A_{\sigma_j}(X_{\ker}); \lambda)=0$. But also  
\begin{align}
\Gamma_{\rm im}(A_{\sigma_j}(X_{\rm im}(X_{\ker}, \lambda)) + A_{\sigma_j}(X_{\ker}); \lambda) & =\Gamma_{\rm im}(A_{\sigma_j}(X_{\rm im}(X_{\ker}, \lambda)+X_{\ker});\lambda)  \nonumber \\ \nonumber = A_{\sigma_j}(\Gamma_{\rm im}(X_{\rm im}(X_{\ker}, &\lambda)+X_{\ker}; \lambda)) = 0\, .
\end{align}
Here, the second equality holds because $\Gamma_{\rm im}=P_{\rm im}\circ \Gamma$ is the composition of $\Sigma$-equivariant maps. By uniqueness of $X_{\rm im}(A_{\sigma_j}(X_{\ker}), \lambda)$, this proves that
$$A_{\sigma_j}(X_{\rm im}(X_{\ker}, \lambda))=X_{\rm im}(A_{\sigma_j}(X_{\rm ker}), \lambda)\, .$$
In other words, the map $X_{\rm im}: \ker L_0^S \times \Lambda \to {\rm im}\, L_0^S$ is $\Sigma$-equivariant.

It now follows easily that $r$ is $\Sigma$-equivariant:
\begin{align}
&r(A_{\sigma_j}(X_{\ker}); \lambda) = \Gamma_{\ker}(X_{\rm im}(A_{\sigma_j}(X_{\ker}), \lambda)+A_{\sigma_{j}}(X_{\ker}); \lambda) = \nonumber \\ \nonumber 
\Gamma_{\ker}(A_{\sigma_j}(&X_{\rm im}(X_{\ker}, \lambda)+X_{\ker}); \lambda) =  A_{\sigma_j}(\Gamma_{\ker}(X_{\rm im}(X_{\ker}, \lambda);\lambda) = A_{\sigma_j}(r(X_{\ker}; \lambda))\, .
\end{align}
Here, the second equality holds because $X_{\rm im}$ is $\Sigma$-equivariant and the third because $\Gamma_{\rm ker}=P_{\ker}\circ \Gamma$ is the composition of $\Sigma$-equivariant maps.
\end{proof}
\begin{remark}
The above construction is only slightly unusual. Indeed, in bifurcation theory, it is perhaps more common to reduce the steady state equation $\Gamma(X;\lambda)=0$ to a bifurcation equation on the kernel $\ker L_0$ of $L_0=D_X\Gamma(X_0; \lambda_0)$. This kernel is a subrepresentation of $W$, but the problem is that it may not be complemented by another subrepresentation. As a result, the reduced equation $r(X_{\ker}; \lambda)=0$ need not be equivariant under this construction. This explains our choice to reduce to an equation on the ``generalized kernel'' $\ker L_0^S$: this kernel is nicely complemented by the subrepresentation ${\rm im}\, L_0^S$, the ``reduced image'' of $L_0$.
\end{remark}

\section{Generic steady state bifurcations}\label{secgeneric}
In this section, we investigate the structure of a generic semigroup equivariant steady state bifurcation in more detail. To this end, assume that $\lambda_0 < \lambda_1$ are real numbers and that 
$$L:(\lambda_0, \lambda_1)\to {\rm End}(W)$$
is a continuously differentiable one-parameter family of endomorphisms. The collection of such curves of endomorphisms is given the $C^1$-topology. 
\begin{definition}
We say that a one-parameter family $\lambda\mapsto L(\lambda)$ is {\it in general position} if for each $\lambda\in (\lambda_0, \lambda_1)$ either $L(\lambda)$ is invertible or the generalized kernel of $L(\lambda)$ is absolutely indecomposable (i.e. of real type).
\end{definition}
We would like to prove that a ``generic'' one-parameter curve of endomorphisms is in general position, because this would imply that steady state bifurcations generically occur along precisely one absolutely indecomposable representation. Instead, we will only prove this result under a  special condition on the representation. It is currently unclear to us if the result is true in more generality. The main result of this section is the following:

\begin{theorem}\label{genericcodimone}
Assume that the representation $W$ of a semigroup splits as a sum of mutually non-isomorphic indecomposable representations. Then the collection of  one-parameter families of endomorphisms in general position is open and dense in the $C^1$-topology.
\end{theorem}
\begin{proof}[Sketch]
Under the prescribed condition on the representation, we will prove below that the set of endomorphisms
$$\{L\in {\rm End}(W) \, |\, \ker L^S \ \mbox{is not absolutely indecomposable}\}$$
is contained in a finite union of submanifolds of ${\rm End}(W)$, each of which has co-dimension at least $2$. Theorem \ref{genericcodimone} therefore follows from the Thom transversality theorem (that implies that every smooth curve can smoothly be perturbed into a curve that does not intersect any given manifold of co-dimension $2$ or higher). 
\end{proof}
The full proof of Theorem \ref{genericcodimone} will be given below in a number of steps, starting with the following preparatory lemma.

\begin{lemma}\label{preplemma}
Let $L_0\in {\rm End}(W)$ and denote by $$W=\ker L^S_0 \oplus {\rm im}\, L^S_0$$ the decomposition into generalized kernel and reduced image of $L_0$, with respect to which  
$$L_0=\left(\begin{array}{cc} L_0^{11}  & 0 \\ 0 & L^{22}_0\end{array}\right)\ \mbox{with}\ L_0^{11}\ \mbox{nilpotent and}\ L_0^{22}\ \mbox{invertible}.$$ 
Then there is an open neighbourhood $U \subset {\rm End}(W)$ of the zero endomorphism and smooth maps $\phi^{11}: U \to {\rm End}(\ker L^S_0)$ and $\phi^{22}: U \to {\rm End}({\rm im}\, L_0^S)$ so that for every $L \in U$, 
$$ L_0+ L \ \mbox{is conjugate to}\ \left(\begin{array}{cc} \phi^{11}(L) & 0 \\ 0 & \phi^{22}(L) \end{array}\right)\, .$$
It holds that $\phi^{11}(L)= L_0^{11}+L^{11} + \mathcal{O}(||L||^2)$ and $\phi^{22}(L)=L^{22}_0+ L^{22} + \mathcal{O}(||L||^2)$.
\end{lemma}
\begin{proof}
Let us define the smooth map $$\Phi: {\rm End}(W)\times \{M\in {\rm End}(W) \, |\, I+M\ \mbox{is invertible}\} \to {\rm End}(W)$$ by  
$$\Phi(L, M) = (I+M) \circ (L_0 + L) \circ (I+M)^{-1}\, .$$
This map admits the Taylor expansion 
\begin{align}\nonumber
 \Phi(L, M)& = (I+M) \circ (L_0+L)\circ (I-M+\mathcal{O}(||M||^2)) \\ \nonumber 
 & = L_0 + L + [M, L_0] + \mathcal{O}(||L||^2 + ||M||^2) \, .
\end{align}
Let us now write $L$ and $M$ in terms of the decomposition $W=\ker L_0^S \oplus {\rm im}\, L_0^S$, i.e. for the appropriate homomorphisms $L^{ij}, M^{ij}$ ($1\leq i,j\leq 2$) we write
$$L= \left(\begin{array}{cc} L^{11}  & L^{12} \\ L^{21} & L^{22}\end{array}\right)\ \mbox{and} \ M= \left(\begin{array}{cc} M^{11}  & M^{12} \\ M^{21} & M^{22}\end{array}\right)  \, .$$
In the same way we denote
$$
\Phi(L,M)= \left(\begin{array}{cc} \Phi^{11}(L,M)  & \Phi^{12}(L,M) \\ \Phi^{21}(L,M) & \Phi^{22}(L,M) \end{array}\right)\, .$$
It is easy to check that

\begin{align}
\nonumber  L+[M,L_0] = \nonumber
 \left(\begin{array}{ll} L^{11} + M^{11} L_0^{11}-L_0^{11}M^{11} & L^{12} +M^{12}L_0^{22} - L_0^{11}M^{12}\\ L^{21} + M^{21}L_0^{11}-L_0^{22}M^{21}& L^{22} + M^{22} L_0^{22} -L_0^{22}M^{22}\end{array}\right)\, .
\end{align}
In other words, 
\begin{align}
&D\Phi^{11}(0, 0)\cdot (L,M) = L^{11} + L_0^{11} M^{11} - M^{11}L_0^{11} \, ,\nonumber \\
&D\Phi^{12}(0, 0)\cdot (L,M) = L^{12} + M^{12}L_0^{22}-L_0^{11}M^{12}\, , \nonumber \\
&D\Phi^{21}(0, 0)\cdot (L,M) = L^{21} + M^{21}L_0^{11}-L_0^{22}M^{21}\, , \nonumber \\
&D\Phi^{22}(0, 0)\cdot (L,M) = L^{22} +M^{22} L_0^{22} - L_0^{22} M^{22} \, .\nonumber
\end{align}
We claim that the operator
$$D_{M^{12}}\Phi^{12}(0,0): M^{12}\mapsto M^{12}L_0^{22}-L_0^{11}M^{12}  \ \mbox{from} \  {\rm Hom}({\rm im} \, L_0^S, \ker L_0^S) \ \mbox{to itself} $$
is invertible. Indeed, the homological equation 
$$M^{12}L_0^{22}-L_0^{11}M^{12} = -L^{12}$$
can be solved for $M^{12}$ to give, for any $N\geq 0$, that
\begin{align}
M^{12}  & =  -L^{12}(L_0^{22})^{-1}  + L_0^{11}M^{12}(L_0^{22})^{-1}     \nonumber \\ \nonumber
 & = -L^{12}(L_0^{22})^{-1}  - L_0^{11} L^{12}(L_0^{22})^{-2} +  (L_0^{11})^{2}M^{12}(L_0^{22})^{-2} \nonumber \\ \nonumber
 & = \ldots = -\sum_{n=0}^{N} (L_0^{11})^n L^{12} (L_0^{22})^{-(n+1)} + (L_0^{11})^{N+1}M^{12}(L_0^{22})^{-(N+1)}\, .
\end{align} 
Because $L_0^{11}$ is nilpotent, $(L_0^{11})^{N+1}=0$ for large enough $N$. This proves our claim that $D_{M^{12}}\Phi^{12}(0,0)$ is invertible. In the same way, $D_{M^{21}}\Phi^{21}(0,0)$ is invertible, and as a consequence so is the map 
$$D_{(M^{12}, M^{21})}(\Phi^{12}, \Phi^{21})(0,0): (M^{12}, M^{21})\mapsto ( M^{12}L_0^{22}-L_0^{11}M^{12},  M^{21}L_0^{11}-L_0^{22}M^{21} )\, . $$ 
This proves, by the implicit function theorem, that there are smooth functions 
$$M^{12}=M^{12}(L, M^{11}, M^{22})\ \mbox{and} \ M^{21}=M^{21}(L, M^{11}, M^{22})\, ,$$ defined for $L, M^{11}$ and $M^{22}$ near zero, such that 
$$\Phi^{12}(L, M(L, M^{11}, M^{22}))=0\ \mbox{and}\ \Phi^{21}(L, M(L, M^{11}, M^{22}))=0\, .$$
If we choose $M^{11}=0$ and $M^{22}=0$, then it actually follows that 
 \begin{align}\nonumber
& \phi^{11}(L):=\Phi^{11}(L, M(L, 0,0)) = L_0^{11}+L^{11} + \mathcal{O}(||L||^2) \ \mbox{and}\\ \nonumber
& \phi^{22}(L):=\Phi^{22}(L, M(L,0,0))=L^{22}_0+ L^{22}+\mathcal{O}(||L||^2)\, .
\end{align}
 This proves the lemma.
\end{proof}

\noindent Let us assume now that the representation $W$ splits as a sum of indecomposables
$$W = W_1\oplus \ldots \oplus W_m\, .$$ 
When $L\in {\rm End}(W)$ is an arbitrary endomorphism, then also $\ker L^S$ is a sum of indecomposable representations. In fact, by the Krull-Schmidt theorem, 
$$\ker L^S\cong W_{i_1}\oplus \ldots \oplus W_{i_k}\ \mbox{with}\ k\leq m \ \mbox{and}\ 1\leq i_1 < i_2 < \ldots < i_k \leq m\, .$$
Thus, we can classify the endomorphisms of $W$ by the isomorphism type of their generalized kernels by defining for all $1\leq i_1 < i_2 < \ldots < i_k \leq m$ the collection
$${\rm Iso}(W_{i_1}\oplus \ldots \oplus W_{i_k}):=\{L\in {\rm End}(W)\, |\, \ker L^S \ \mbox{is isomorphic to} \  W_{i_1}\oplus \ldots \oplus W_{i_k}\}\, .$$

\noindent The following proposition gives a description of ${\rm Iso}(W_{i_1}\oplus \ldots \oplus W_{i_k})$ in case $W_{i_1}, \ldots, W_{i_k}$ are mutually nonisomorphic:

\begin{proposition}\label{isoprop}
Suppose $W$ splits as the sum of indecomposables $W_1\oplus \ldots \oplus W_m$ and let 
$${\rm ind}\, W_i := \dim {\rm End}(W_i)/{\rm End}^{\rm Nil}(W_i) \ (= 1,2 \ \mbox{or}\ 4)\ \mbox{for} \ 1\leq i \leq m\, .$$
Choose $1\leq i_1 < i_2 < \ldots < i_k \leq m$ and assume that the $W_{i_j}$ are mutually nonisomorphic. 
Then ${\rm Iso}(W_{i_1}\oplus \ldots \oplus W_{i_k})$ is contained in a submanifold of ${\rm End}(W)$ of co-dimension 
$${\rm ind}\, W_{i_1} + \ldots + {\rm ind}\, W_{i_k}\, .$$
\end{proposition}
\begin{proof}
Choose an arbitrary isomorphism $L_0\in {\rm Iso}(W_{i_1}\oplus \ldots \oplus W_{i_k})$ and recall from Lemma \ref{preplemma} that an endomorphism $L_0+L$ close to $L_0$ is conjugate to
$$\left(\begin{array}{cc} \phi^{11}(L) & 0 \\ 0 & \phi^{22}(L) \end{array}\right)\, .$$ 
It is clear that $\phi^{22}(L)$ is invertible, because $\phi^{22}(L)=L_0^{22}+\mathcal{O}(||L||)$. The generalized kernels of $L_0$ and $L_0+L$ are therefore isomorphic if and only if $\phi^{11}(L)$ is nilpotent. 

Moreover, $\phi^{11}$ is a submersion at $L_0$ because $\phi^{11}(L)=L_0^{11}+L^{11}+\mathcal{O}(||L||^2)$. By the submersion theorem it thus suffices to check that  
$${\rm End}^{\rm Nil}(W_{i_1} \oplus \ldots \oplus W_{i_k})=\{B\in {\rm End}(W_{i_1}\oplus \ldots \oplus W_{i_k}) \, |\, B\ \mbox{is nilpotent}\, \}$$
is contained in a submanifold of ${\rm End}(W_{i_1}\oplus \ldots \oplus W_{i_k})$ of the prescribed co-dimension. 

To check this fact, let us decompose an arbitrary $B\in {\rm End}(W_{i_1} \oplus \ldots \oplus W_{i_k})$ as
$$B= \left( \begin{array}{cccc} B^{11} & B^{12} & \cdots & B^{1k} \\ B^{21} & B^{22} & \cdots & B_{2k}\\ \vdots & \vdots & \ddots & \vdots \\ B^{k1} & B^{k2} & \cdots & B^{kk}  \end{array} \right) $$
with $B^{jl}\in {\rm Hom}(W_{i_j}, W_{i_l})$. Then it holds for any $n\geq 1$ that 
$$(B^n)^{jl} = \!\!\! \sum_{1\leq k_1, \ldots, k_{n-1} \leq k} \!\!\! B^{j k_1} B^{k_1k_2} \cdots B^{k_{n-1} l}\, .$$
We now remark that any composition 
$$B^{j k_1} B^{k_1k_2} \cdots B^{k_{n-1} j} \in {\rm End}(W_{i_j})$$
is nilpotent as soon as there exists an $r$ with $k_r\neq j$: otherwise, by Proposition \ref{oneev}, it would have been an isomorphism and by Proposition \ref{twoisos}, then $W_{i_j}$ would have been isomorphic to $W_{i_{k_r}}$, which contradicts our assumptions. It therefore follows from Proposition \ref{sumnilpotent} that
$$(B^n)^{jj} = (B^{jj})^n + ``{\rm nilpotent}"\, .$$
Assume now that $B$ is nilpotent. Then there is an $n$ so that $B^n=0$. For this $n$ it then holds that $(B^{jj})^n$ is nilpotent and hence that $B^{jj}$ is nilpotent. This finishes the proof that whenever $W_{i_1}, \ldots, W_{i_k}$ are mutually nonisomorphic indecomposable representations, then 
$${\rm End}^{\rm Nil}(W_{i_1} \oplus \ldots \oplus W_{i_k}) \subset \{B\in {\rm End}(W_{i_1} \oplus \ldots \oplus W_{i_k})\, |\, B^{jj}\in {\rm End}^{\rm Nil}(W_{i_j})\ \mbox{for all}\ 1\leq j\leq k\, \}\, .$$
The latter is a subspace (and in particular a submanifold) of ${\rm End}(W_{i_1}\oplus \ldots \oplus W_{i_k})$ of co-dimension ${\rm ind}\, W_{i_1}+ \ldots + {\rm ind}\, W_{i_k}$. This proves the proposition. 
\end{proof}
 
\begin{remark}
If $W_i$ is one of the indecomposable factors of $W$, then ${\rm End}^{\rm Nil}(W_i)\subset {\rm End}(W_i)$ is a linear subspace (it is not just contained in one). This implies that ${\rm Iso}(W_i)$ is a true submanifold of ${\rm End}(W)$ (and not just contained in one).
The co-dimension of this submanifold is $1$ if $W_i$ is of real type, $2$ if $W_i$ is of complex type and $4$ if $W_i$ is of quaternionic type.
\end{remark}
 \noindent We are now ready to finish the proof of Theorem \ref{genericcodimone}. 
 
 \begin{proofof}\! \!\!\! [of Theorem \ref{genericcodimone}]: Recall that the co-dimension of ${\rm Iso}(W_{i_1}\oplus \ldots \oplus W_{i_k})$ is equal to 
 $${\rm ind} \, W_{i_1} + \ldots + {\rm ind}\,  W_{i_k}\, .$$
 In particular, this co-dimension is equal to zero if and only if $k=0$ and is equal to one if and only if $k=1$ and $W_{i_1}$ is absolutely indecomposable.
  This proves that the collection $$\{L\in {\rm End}(W) \, |\, \ker L^S \ \mbox{is not absolutely indecomposable}\}$$
is contained in the union of finitely many submanifolds of ${\rm End}(W)$, each of which has co-dimension $2$ or higher. The Thom transversality theorem  finishes the proof. 
\end{proofof}
\noindent For completeness, let us state here as an obvious corollary of Theorem \ref{genericcodimone} that a generic co-dimension one synchrony breaking steady state bifurcation must occur along an absolutely indecomposable representation:
\begin{corollary}
Let $W$ be a representation of a semigroup $\Sigma$ and assume that it splits as a sum of mutually non-isomorphic indecomposable representations. Moreover, let $X_0\in W$ be  $\Sigma$-symmetric, i.e. $A_{\sigma_j}(X_0)=X_0$ for all $1\leq j\leq n$. 

We define the set of curves of equivariant vector fields admitting $X_0$ as an equilibrium by
$$E:=\{ \Gamma: W\times (\lambda_0, \lambda_1) \to W\, | \  \Gamma\ \mbox{is smooth, commutes with}\ \Sigma\ \mbox{and}\ \Gamma(X_0; \lambda)=0 \, \}$$
and the subset of those curves that are in general position as
$$E_{\rm gp}:=\left\{\, \Gamma\in E\, |\,  \ker D_{X}\Gamma(X_0; \lambda)\ \mbox{is either trivial or absolutely indecomposable}\, \right\}\, .$$ 
Then it holds that $E_{\rm gp}$ is open and dense in $E$ in the $C^1$-topology.
\end{corollary}
\begin{proof}
Obvious from Theorem \ref{genericcodimone}. 
\end{proof}

\section{Monoid networks with two or three cells}\label{sectwoorthree}
In this section, we investigate the steady state bifurcations that can occur in fundamental monoid networks with two or three cells, where for simplicity we let $V$ be one-dimensional. It so turns out that for all these networks, the corresponding semigroup representations split as the sum of mutually nonisomorphic indecomposable representations. Thus, we are able to classify all possible generic co-dimension one steady state bifurcations in these networks.

\subsection{Monoid networks with two cells}
It is clear that up to isomorphism there are precisely two monoids with two elements, say $\Sigma_1$ and $\Sigma_2$, with multiplication tables
$$ \begin{array}{c|cc}\Sigma_1 & \sigma_1 & \sigma_2\\ \hline 
\sigma_1 & \sigma_1 & \sigma_2 \\
\sigma_2 & \sigma_2 & \sigma_1
\end{array}  \ \mbox{and} \ 
\begin{array}{c|cc}\Sigma_2 & \sigma_1 & \sigma_2\\ \hline 
\sigma_1 & \sigma_1 & \sigma_2 \\
\sigma_2 & \sigma_2 & \sigma_2
\end{array}\, .
$$
Below, we shall investigate their fundamental networks separately.
\subsubsection*{Bifurcations for $\Sigma_1$}
The monoid $\Sigma_1$ is the group $\Z_2$ and the corresponding semigroup representation is given by
\begin{align}\nonumber
A_{\sigma_1}(X_1, X_2)=(X_1, X_2)\, , \\ \nonumber
A_{\sigma_2}(X_1, X_2)=(X_2, X_1)\, .
\end{align}
In particular, the fundamental network is given by the differential equations 
 \begin{align}\nonumber
 \begin{array}{rl}
 \dot X_1 &= f(X_1, X_2) \, ,\\
 \dot X_2 &= f(X_2, X_1)\, . \end{array} 
\end{align}
The bifurcation theory of such equivariant networks is of course well-known, but we summarize it here for completeness.

First of all, the indecomposable decomposition of the phase space is a unique decomposition into mutually nonisomorphic irreducible representations of $\Sigma_1$, given by 
$$\{X_1=X_2\} \oplus \{X_1+X_2=0\}\, .$$
The subrepresentation $\{X_1=X_2\}$ is trivial in the sense that $A_{\sigma_2}$ acts upon it as the identity. Thus, if we use $X_1$ as a coordinate on this representation, the resulting bifurcation equation after Lyapunov-Schmidt reduction must be of the form $r(X_1;\lambda)=0$ for a function $r(X_1; \lambda)$ satisfying $r(0;0)=0$, i.e.
$$r(X_1;\lambda)=a \lambda + b X_1^2 + \mathcal{O}(|\lambda|^2 + |\lambda|\cdot |X_1| + |X_1|^3)\, .$$
Under the generic conditions that $a, b\neq 0$, the solutions of the bifurcation equation are of the form
$$X_1=X_2=\pm\sqrt{-(a/b)\lambda}+\mathcal{O}(\lambda)\, .$$
We conclude that, generically, a synchronous saddle-node bifurcation takes place along the trivial subrepresentation.

The subrepresentation $\{X_1+X_2=0\}$ is acted upon by $A_{\sigma_2}$ as minus identity. Choosing again $X_1$ as a coordinate, this yields an equivariant bifurcation equation of the form $r(X_1; \lambda)=0$ with $r(-X_1;\lambda)=-r(X_1; \lambda)$, i.e.
$$r(X_1; \lambda) = a \lambda X_1 + b X_1^3 + \mathcal{O}(|\lambda|^2\cdot |X_1| + |\lambda|\cdot |X_1|^3 + |X_1|^5)\, .$$
Under the generic conditions that $a, b\neq 0$, this yields a pitchfork bifurcation, i.e. solutions are of the form
$$X_1=X_2=0\ \mbox{or}\ X_1=-X_2 = \pm \sqrt{-(a/b)\lambda} + \mathcal{O}(\lambda)\, .$$
\subsubsection*{Bifurcations for $\Sigma_2$.}
We first of all remark that the monoid $\Sigma_2$ is not a group. Its representation is given by 
\begin{align}\nonumber
A_{\sigma_1}(X_1, X_2)=(X_1, X_2)\, , \\ \nonumber
A_{\sigma_2}(X_1, X_2)=(X_2, X_2)\, ,
\end{align}
and the corresponding differential equations read
$$
  \begin{array}{rl}
 \dot X_1 &= f(X_1, X_2) \, ,\\
 \dot X_2 &= f(X_2, X_2) \, .
\end{array}
$$
Again, the indecomposable decomposition of the representation is a unique decomposition into mutually nonisomorphic irreducible representations, now given by 
$$\{X_1=X_2\}\oplus \{X_2=0\}\, .$$
The subrepresentation $\{X_1=X_2\}$ is trivial so that once more only a synchronous saddle-node bifurcation is expected along this representation.

The subrepresentation $\{X_2=0\}$ is acted upon by $A_{\sigma_2}$ as the zero map though. Equivariance of the reduced bifurcation equation $r(X_1; \lambda)=0$ under the map $X_1\mapsto 0$ just means that $r(0;\lambda)=0$ and thus that 
$$r(X_1; \lambda) = a \lambda X_1 + b X_1^2 + \mathcal{O}(|\lambda|\cdot |X_1|^2 + |X_1|^3)\, .$$
Under the generic conditions that $a, b \neq 0$, this produces a transcritical bifurcation with solution branches
$$X_1=X_2=0  \ \mbox{and}\ X_1=0, X_2 = -(a/b)\lambda + \mathcal{O}(\lambda^2)\, .$$
Interestingly, the transcritical bifurcation arises here as a generic co-dimension one equivariant bifurcation.

\subsection{Monoid networks with three cells}
Up to isomorphism, there are precisely $7$ monoids with three elements. Their multiplication tables are the following:
\begin{align}\nonumber
& \begin{array}{c|ccc}\Sigma_1 & \sigma_1 & \sigma_2 & \sigma_3\\ \hline 
\sigma_1 & \sigma_1 & \sigma_2  & \sigma_3\\
\sigma_2 & \sigma_2 & \sigma_2 & \sigma_2 \\
\sigma_3 & \sigma_3 & \sigma_2 & \sigma_2
\end{array} \ \ 
\begin{array}{c|ccc}\Sigma_2 & \sigma_1 & \sigma_2 & \sigma_3\\ \hline 
\sigma_1 & \sigma_1 & \sigma_2  & \sigma_3\\
\sigma_2 & \sigma_2 & \sigma_2 & \sigma_3 \\
\sigma_3 & \sigma_3 & \sigma_3 & \sigma_2
\end{array} \ \ 
\begin{array}{c|ccc}\Sigma_3 & \sigma_1 & \sigma_2 & \sigma_3\\ \hline 
\sigma_1 & \sigma_1 & \sigma_2  & \sigma_3\\
\sigma_2 & \sigma_2 & \sigma_2 & \sigma_3 \\
\sigma_3 & \sigma_3 & \sigma_3 & \sigma_3
\end{array} \ \ 
\begin{array}{c|ccc}\Sigma_4 & \sigma_1 & \sigma_2 & \sigma_3\\ \hline 
\sigma_1 & \sigma_1 & \sigma_2  & \sigma_3\\
\sigma_2 & \sigma_2 & \sigma_2 & \sigma_2 \\
\sigma_3 & \sigma_3 & \sigma_3 & \sigma_3
\end{array} 
\\ \nonumber
& \begin{array}{c|ccc}\Sigma_5 & \sigma_1 & \sigma_2 & \sigma_3\\ \hline 
\sigma_1 & \sigma_1 & \sigma_2  & \sigma_3\\
\sigma_2 & \sigma_2 & \sigma_2 & \sigma_3 \\
\sigma_3 & \sigma_3 & \sigma_2 & \sigma_3
\end{array} \ \ 
\begin{array}{c|ccc}\Sigma_6 & \sigma_1 & \sigma_2 & \sigma_3\\ \hline 
\sigma_1 & \sigma_1 & \sigma_2  & \sigma_3\\
\sigma_2 & \sigma_2 & \sigma_3 & \sigma_1 \\
\sigma_3 & \sigma_3 & \sigma_1 & \sigma_2
\end{array} \ \ 
\begin{array}{c|ccc}\Sigma_7 & \sigma_1 & \sigma_2 & \sigma_3\\ \hline 
\sigma_1 & \sigma_1 & \sigma_2  & \sigma_3\\
\sigma_2 & \sigma_2 & \sigma_1 & \sigma_3 \\
\sigma_3 & \sigma_3 & \sigma_3 & \sigma_3
\end{array} \, .
\end{align}
Below, we shall investigate the steady state bifurcations in the corresponding fundamental networks separately:

\subsubsection*{Bifurcations for $\Sigma_1$}
We have graphically depicted $\Sigma_1$ in Figure \ref{pict3} below.
\begin{figure}[ht]\renewcommand{\figurename}{\rm \bf \footnotesize Figure}
\centering
\begin{tikzpicture}[->,>=stealth',shorten >=1pt,auto,node distance=1.5cm,
                    thick,main node/.style={circle,draw,font=\sffamily\large\bfseries}]

  \node[main node] (1) {$x_1$};
  \node[main node] (2) [below of=1] {$x_2$};
  \node[main node] (3) [below of=2] {$x_3$};

   \node[main node] (4) [right of=1] {$x_1$};
  \node[main node] (5) [below of=4] {$x_2$};
  \node[main node] (6) [below of=5] {$x_3$};
  \node[main node] (7) [right of =4] {$x_1$};
  \node[main node] (8) [below  of=7] {$x_2$};
  \node[main node] (9) [below of=8] {$x_3$};

  \path[every node/.style={font=\sffamily\small}]
(1) edge [loop below] node {} (1)
(2) edge [loop below] node {} (2)
(3) edge [loop below] node {} (3)

(5) edge node {} (4)
(5) edge [loop below] node {} (5)
(5) edge [bend right] node {} (6)

(9) edge [bend right] node {} (7)
(8) edge [loop below] node {} (8)
(8) edge [bend right] node {} (9)
    ;
\end{tikzpicture}

\caption{\footnotesize {\rm The collection $\Sigma_1$ depicted as a directed multigraph.}}
\label{pict3}
\end{figure}
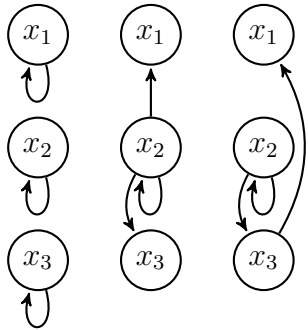

The representation of $\Sigma_1$ is given by 
  \begin{align}\nonumber
\begin{array}{rl}
A_{\sigma_1}(X) &= (X_1, X_2, X_3) \, ,\\
A_{\sigma_2}(X) &= (X_2, X_2, X_2) \, ,\\
A_{\sigma_3}(X) &= (X_3, X_2, X_2)\, .
\end{array} 
 \end{align}
 This representation uniquely splits as a sum of mutually nonisomorphic indecomposables
$$ \{X_1=X_2=X_3\} \oplus \{X_2=0\}\, .$$ 
The action of $\Sigma_1$ on the subrepresentation $\{X_1=X_2=X_3\}$ is trivial, so only a saddle-node bifurcation can generically occur along this irreducible representation.

On the indecomposable subrepresentation $\{X_2=0\}$, let us choose coordinates $(X_1, X_3)$. Then the action of $A_{\sigma_2}$ and $A_{\sigma_3}$ on this subrepresentation is given by 
\begin{align}\nonumber
\begin{array}{rl}
A_{\sigma_2}(X_1, X_3) &= (0, 0) \, ,\\
A_{\sigma_3}(X_1, X_3) &= (X_3, 0)\, .
\end{array} 
 \end{align}
This confirms that $\{X_2=0\}$ is indecomposable, but not irreducible, because it contains the one-dimensional subrepresentation $\{X_2=X_3=0\}$. Moreover, one computes that 
$${\rm End}(\{X_2=0\})=\{(X_1, X_3)\mapsto (\alpha X_1+\beta X_3, \alpha X_3)\, |\, \alpha, \beta \in \R\}\, . $$
This shows that $\{X_2=0\}$ is absolutely indecomposable (i.e. of real type), and that there exist nontrivial nilpotent endomorphisms, namely of the form $(X_1, X_3)\mapsto (\beta X_3, 0)$.

The bifurcation equation $r(X_1, X_3; \lambda) = (r_1(X_1, X_3; \lambda), r_3(X_1, X_3; \lambda)) = (0,0)$ is equivariant precisely when
$$r_{1}(0, 0; \lambda)=0, r_{3}(0, 0; \lambda)=0, r_1(X_3,0;\lambda) = r_{3}(X_1, X_3; \lambda) \ \mbox{and}\  r_3(X_3,0;\lambda) = 0 \, .$$
This implies that
\begin{align}\nonumber
r_1(X_1, X_3; \lambda)& =  a \lambda X_1 + bX_3 + cX_1^2 + \mathcal{O}(|\lambda|^2\cdot |X_1|+|\lambda|\cdot |X_1|^2 + |X_1|^3+|\lambda|\cdot |X_3|)\, , \\ \nonumber 
 r_3(X_1, X_3; \lambda)& =  a \lambda X_3 + cX_3^2  + \mathcal{O}(|\lambda|^2\cdot |X_3|+ |\lambda|\cdot |X_3|^2+|X_3|^3) \, .
\end{align} 
Under the generic conditions that $a, b, c\neq 0$, this gives three solution branches:
\begin{align}\nonumber
\begin{array}{lll} X_1=0, & X_2=0, & X_3=0,\\
X_1= -(a/c)\lambda + \mathcal{O}(\lambda^2), & X_2= 0,& X_3= 0, \\ \nonumber
X_1=\pm\sqrt{(ab/c^2)\lambda} + \mathcal{O}(\lambda),& X_2 = 0, & X_3= -(a/c)\lambda + \mathcal{O}(\lambda^2).
\end{array}
\end{align}
This means that a fully synchronous trivial branch, a partially synchronous transcritical branch and a fully nonsynchronous saddle-node branch coalesce in this bifurcation. We note that this phenomenon was observed before for this network in \cite{feedforwardRinkSanders} and \cite{CCN}. A diagram of this bifurcation is given in Figure \ref{figbif}.
 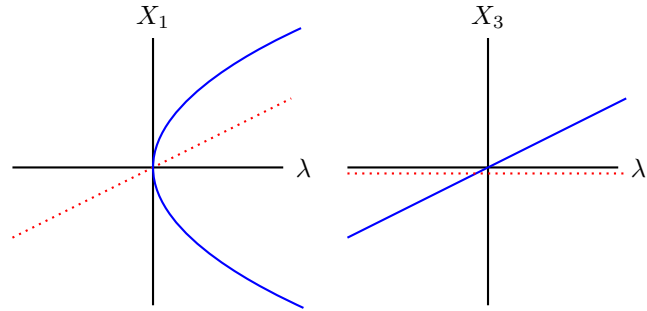
\begin{figure}[ht]\renewcommand{\figurename}{\rm \bf \footnotesize Figure}
\centering
\begin{tikzpicture} [>=stealth',shorten >=1pt,auto,node distance=3cm,
                    thick,main node/.style={draw,font=\sffamily\bfseries}]
 \node (a) at (-2,0) {};
 \node (a1) at (-2, -.08){};
 \node (b1) at (2,-.08){};
 \node (x) at (0,2) {$X_1$};
 \node (b) at (2,0) {$\lambda$};
 \node (c) at (-2,-1) {};
 \node (g) at (-2,-1){} ;
 \node (d) at (2,1) {};
 \node (h) at (2,1) {};
 \node (e) at (2,2) {};
 \node (f) at (2,-2) {};saddle-node
 \node (j) at (0,-2) {};
 \draw (x)--(j);
  \draw (a) --(b);
  \draw[dotted, red] (g) --(h);
  \draw[rotate=270, blue] (f.north) parabola bend (0,0) (e.south);
\end{tikzpicture}
 \begin{tikzpicture} [>=stealth',shorten >=1pt,auto,node distance=3cm,
                    thick,main node/.style={draw,font=\sffamily\bfseries}]
 \node (a) at (-2,0) {};
 \node (x) at (0,1) {};
 \node (b) at (2,0) {$\lambda$};
 \node (c) at (-2,-1) {};
 \node (g) at (-2,-0.08){} ;
 \node (g2) at (-2, -.15){};
 \node (d) at (2,1) {};
 \node (h) at (2,-0.08) {};
 \node (h2) at (2,-.15){};
 \node (e) at (2,2) {};
 \node (f) at (2,-2) {};
 \node (i) at (0,2) {$X_3$};
 \node (j) at (0,-2) {};
\draw (i)--(j);
  \draw (a) --(b);
  \draw[blue] (c) --(d);
  \draw[dotted, red] (g) --(h);
\end{tikzpicture}
\caption{\footnotesize {\rm Bifurcation diagram of a co-dimension one steady state bifurcation in the fundamental network of $\Sigma_1$. This figure depicts the nontrivial solution branches in case $a, b>0$ and $c<0$.}}
\label{figbif}
\end{figure}

\subsubsection*{Bifurcations for $\Sigma_2$}
We have graphically depicted $\Sigma_2$ in Figure \ref{pict4} below.
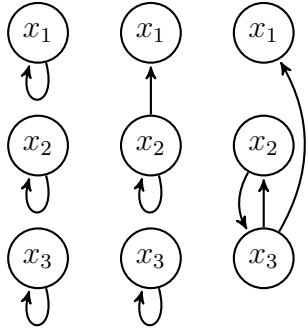
\begin{figure}[ht]\renewcommand{\figurename}{\rm \bf \footnotesize Figure}
\centering
\begin{tikzpicture}[->,>=stealth',shorten >=1pt,auto,node distance=1.5cm,
                    thick,main node/.style={circle,draw,font=\sffamily\large\bfseries}]

  \node[main node] (1) {$x_1$};
  \node[main node] (2) [below of=1] {$x_2$};
  \node[main node] (3) [below of=2] {$x_3$};

   \node[main node] (4) [right of=1] {$x_1$};
  \node[main node] (5) [below of=4] {$x_2$};
  \node[main node] (6) [below of=5] {$x_3$};
  \node[main node] (7) [right of =4] {$x_1$};
  \node[main node] (8) [below  of=7] {$x_2$};
  \node[main node] (9) [below of=8] {$x_3$};

  \path[every node/.style={font=\sffamily\small}]
(1) edge [loop below] node {} (1)
(2) edge [loop below] node {} (2)
(3) edge [loop below] node {} (3)

(5) edge node {} (4)
(5) edge [loop below] node {} (5)
(6) edge [loop below] node {} (6)

(9) edge node {} (8)
(9) edge [bend right] node {} (7)
(8) edge [bend right] node {} (9)
    ;

\end{tikzpicture}
\caption{\footnotesize {\rm The collection $\Sigma_2$ depicted as a directed multigraph.}}
\label{pict4}
\end{figure}

The representation of $\Sigma_2$ is given by 
  \begin{align}\nonumber
\begin{array}{rl}
A_{\sigma_1}(X) &= (X_1, X_2, X_3) \, ,\\
A_{\sigma_2}(X) &= (X_2, X_2, X_3) \, ,\\
A_{\sigma_3}(X) &= (X_3, X_3, X_2)\, .
\end{array} 
 \end{align}
 This representation uniquely splits as a sum of mutually nonisomorphic one-dimensional irreducible representations
$$ \{X_1=X_2=X_3\}\oplus \{X_2=X_3=0\}\oplus \{X_1=X_2=-X_3\}\, .$$ 
The action of $\Sigma_2$ on the subrepresentation $\{X_1=X_2=X_3\}$ is trivial, so only a saddle-node bifurcation can generically occur along this irreducible representation.

On the subrepresentation $\{X_2=X_3=0\}$, both $A_{\sigma_2}$ and $A_{\sigma_3}$ act as the zero map. Thus one expects a transcritical bifurcation to occur along this irreducible representation.

On the subrepresentation $\{X_1=X_2=-X_3\}$, the map $A_{\sigma_2}$ acts as identity, while $A_{\sigma_3}$ acts as minus identity. This means that a generic steady state bifurcation along this irreducible representation must be a pitchfork bifurcation.

\subsubsection*{Bifurcations for $\Sigma_3$}
We have graphically depicted $\Sigma_3$ in Figure \ref{pict5} below.
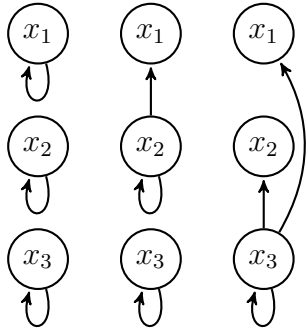
\begin{figure}[ht]\renewcommand{\figurename}{\rm \bf \footnotesize Figure}
\centering
\begin{tikzpicture}[->,>=stealth',shorten >=1pt,auto,node distance=1.5cm,
                    thick,main node/.style={circle,draw,font=\sffamily\large\bfseries}]

  \node[main node] (1) {$x_1$};
  \node[main node] (2) [below of=1] {$x_2$};
  \node[main node] (3) [below of=2] {$x_3$};

   \node[main node] (4) [right of=1] {$x_1$};
  \node[main node] (5) [below of=4] {$x_2$};
  \node[main node] (6) [below of=5] {$x_3$};
  \node[main node] (7) [right of =4] {$x_1$};
  \node[main node] (8) [below  of=7] {$x_2$};
  \node[main node] (9) [below of=8] {$x_3$};

  \path[every node/.style={font=\sffamily\small}]
(1) edge [loop below] node {} (1)
(2) edge [loop below] node {} (2)
(3) edge [loop below] node {} (3)

(5) edge node {} (4)
(5) edge [loop below] node {} (5)
(6) edge [loop below] node {} (6)

(9) edge node {} (8)
(9) edge [bend right] node {} (7)
(9) edge [loop below] node {} (9)
    ;

\end{tikzpicture}
\caption{\footnotesize {\rm The collection $\Sigma_3$ depicted as a directed multigraph.}}
\label{pict5}
\end{figure}

\noindent The representation of $\Sigma_3$ is given by 
  \begin{align}\nonumber
\begin{array}{rl}
A_{\sigma_1}(X) &= (X_1, X_2, X_3) \, ,\\
A_{\sigma_2}(X) &= (X_2, X_2, X_3) \, ,\\
A_{\sigma_3}(X) &= (X_3, X_3, X_3)\, .
\end{array} 
 \end{align}
 This representation uniquely splits as a sum of mutually nonisomorphic one-dimensional irreducible representations
$$  \{X_1=X_2=X_3\} \oplus \{X_2=X_3=0\} \oplus \{X_1=X_2, X_3=0\} \, .$$ 
The action of $\Sigma_3$ on the subrepresentation $\{X_1=X_2=X_3\}$ is trivial, so only a saddle-node bifurcation can generically occur along this irreducible representation.

On the subrepresentation $\{X_2=X_3=0\}$, both $A_{\sigma_2}$ and $A_{\sigma_3}$ act as the zero map. Thus one expects a transcritical bifurcation to occur along this irreducible representation.

On the subrepresentation $\{X_1=X_2, X_3=0\}$, the map $A_{\sigma_2}$ acts as identity, while $A_{\sigma_3}$ acts as the zero map. Hence, a transcritical bifurcation must occur generically along this irreducible representation as well.

\subsubsection*{Bifurcations for $\Sigma_4$}
We have graphically depicted $\Sigma_4$ in Figure \ref{pict6} below.
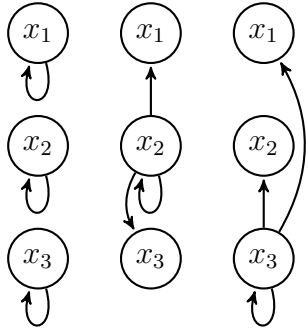
\begin{figure}[ht]\renewcommand{\figurename}{\rm \bf \footnotesize Figure}
\centering
\begin{tikzpicture}[->,>=stealth',shorten >=1pt,auto,node distance=1.5cm,
                    thick,main node/.style={circle,draw,font=\sffamily\large\bfseries}]

  \node[main node] (1) {$x_1$};
  \node[main node] (2) [below of=1] {$x_2$};
  \node[main node] (3) [below of=2] {$x_3$};

   \node[main node] (4) [right of=1] {$x_1$};
  \node[main node] (5) [below of=4] {$x_2$};
  \node[main node] (6) [below of=5] {$x_3$};
  \node[main node] (7) [right of =4] {$x_1$};
  \node[main node] (8) [below  of=7] {$x_2$};
  \node[main node] (9) [below of=8] {$x_3$};

  \path[every node/.style={font=\sffamily\small}]
(1) edge [loop below] node {} (1)
(2) edge [loop below] node {} (2)
(3) edge [loop below] node {} (3)

(5) edge node {} (4)
(5) edge [loop below] node {} (5)
(5) edge [bend right] node {} (6)

(9) edge node {} (8)
(9) edge [bend right] node {} (7)
(9) edge [loop below] node {} (9)
    ;

\end{tikzpicture}
\caption{\footnotesize {\rm The collection $\Sigma_4$ depicted as a directed multigraph.}}
\label{pict6}
\end{figure}

\noindent The representation of $\Sigma_4$ is given by 
  \begin{align}\nonumber
\begin{array}{rl}
A_{\sigma_1}(X) &= (X_1, X_2, X_3) \, ,\\
A_{\sigma_2}(X) &= (X_2, X_2, X_3) \, ,\\
A_{\sigma_3}(X) &= (X_3, X_2, X_3)\, .
\end{array} 
 \end{align}
 This representation nonuniquely splits as a sum of mutually nonisomorphic indecomposables
$$\{X_1=X_2=X_3\} \oplus \{(1+a) X_2 + (1-a) X_3=0\}\ \mbox{for} \ a\in \R \, .$$ 
The action of $\Sigma_4$ on the subrepresentation $\{X_1=X_2=X_3\}$ is trivial, so only a saddle-node bifurcation can generically occur along this irreducible representation.

Because the two-dimensional indecomposable representations $\{(1+a) X_2 + (1-a) X_3=0\}$ are all isomorphic by the Krull-Schmidt theorem, let us consider only $\{X_3=0\}$ and choose coordinates $(X_1, X_2)$. The action of $A_{\sigma_2}$ and $A_{\sigma_3}$ on this subrepresentation then reads
\begin{align}\nonumber
\begin{array}{rl}
A_{\sigma_2}(X_1, X_2) &= (X_2, X_2) \, ,\\
A_{\sigma_3}(X_1, X_2) &= (0, X_2)\, .
\end{array} 
 \end{align}
This confirms that $\{X_3=0\}$ is indecomposable, but not irreducible, because it contains the one-dimensional subrepresentation $\{X_2=X_3=0\}$. Moreover, one computes that 
$${\rm End}(\{X_3=0\})=\{ \alpha I \, |\, \alpha \in \R\}\, . $$
This shows that $\{X_3=0\}$ is absolutely indecomposable (i.e. of real type), and that there do not exist nontrivial nilpotent endomorphisms.

The bifurcation equation $r(X_1, X_2; \lambda) = (r_1(X_1, X_2; \lambda), r_2(X_1, X_2; \lambda)) = (0,0)$ is equivariant precisely when
\begin{align}\nonumber
&r_{1}(X_2, X_2; \lambda)= r_{2}(X_2, X_2; \lambda)=r_2(X_1, X_2; \lambda), \\ \nonumber 
&r_1(0,X_2;\lambda) = 0 \ \mbox{and}\  r_2(0, X_2;\lambda) = r_2(X_1, X_2;\lambda) \, .
\end{align}
These conditions imply that
\begin{align}\nonumber
r_1(X_1, X_2; \lambda)& =  a \lambda X_1 + bX_1X_2 + cX_1^2 \\ \nonumber&+ \mathcal{O}\left(|X_1|\cdot( |\lambda|^2 + |\lambda|\cdot|X_1| +  |\lambda|\cdot|X_2|+ |X_1|^2+ |X_2|^2)\right)\, , \\ \nonumber 
 r_2(X_1, X_2; \lambda)& =  a \lambda X_2 + (b+c) X_2^2  + \mathcal{O}(|\lambda|^2\cdot|X_2|+|\lambda|\cdot |X_2|^2+|X_2|^3) \, .
\end{align} 
Under the generic conditions that $a, b+c, c\neq 0$, this gives four solution branches:
\begin{align}\nonumber
\begin{array}{lll} X_1=0, & X_2=0, & X_3=0,\\
X_1=0, & X_2= -\frac{a}{b+c}\lambda + \mathcal{O}(\lambda^2), & X_3=0,\\
X_1= -\frac{a}{c}\lambda + \mathcal{O}(\lambda^2), & X_2= 0,& X_3= 0, \\ \nonumber
X_1=  -\frac{a}{b+c}\lambda + \mathcal{O}(\lambda^2), & X_2 =  -\frac{a}{b+c}\lambda + \mathcal{O}(\lambda^2), & X_3=0.
\end{array}
\end{align}
This means that in this bifurcation a fully synchronous branch and three partially synchronous transcritical branches come together. A diagram of this bifurcation is given in Figure \ref{figbif2}.

 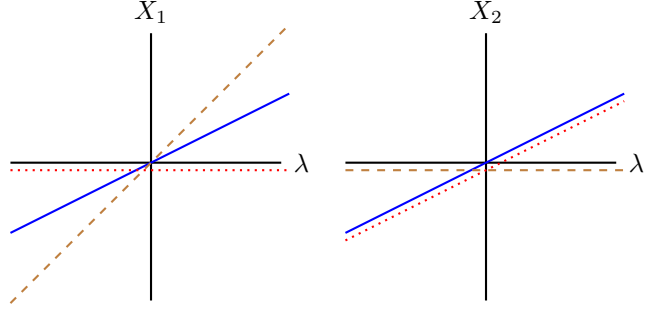
\begin{figure}[ht]\renewcommand{\figurename}{\rm \bf \footnotesize Figure}
\centering
\begin{tikzpicture} [>=stealth',shorten >=1pt,auto,node distance=3cm,
                    thick,main node/.style={draw,font=\sffamily\bfseries}]
 \node (a) at (-2,0) {};
 \node (x) at (0,2) {$X_1$};
 \node (b) at (2,0) {$\lambda$};
 \node (m) at (0,-2){};
 \node (c) at (-2,-1) {};
 \node (k) at (-2,-2) {};
 \node (g) at (-2,-0.1){} ;
 \node (d) at (2,1) {};
 \node (l) at (2,2) {};
 \node (h) at (2,-0.1) {};
 \node (e) at (2,2) {};
 \node (f) at (2,-2) {};
 \draw (x)--(m);
  \draw (a) --(b);
  \draw[blue] (c) --(d);
  \draw[dashed, brown] (k) --(l);
  \draw[dotted, red] (g) --(h);
\end{tikzpicture}
\begin{tikzpicture} [>=stealth',shorten >=1pt,auto,node distance=3cm,
                    thick,main node/.style={draw,font=\sffamily\bfseries}]
 \node (a) at (-2,0) {};
 \node (m) at (0,-2){};
 \node (x) at (0,2) {$X_2$};
 \node (b) at (2,0) {$\lambda$};
 \node (c) at (-2,-1) {};
 \node (k) at (-2,-1.1) {};
 \node (g) at (-2,-0.1){} ;
 \node (d) at (2,1) {};
 \node (l) at (2,0.9) {};
 \node (h) at (2,-0.1) {};
 \node (e) at (2,2) {};
 \node (f) at (2,-2) {};
 \draw (x)--(m);
  \draw (a) --(b);
  \draw[blue] (c) --(d);
  \draw[dotted, red] (k) --(l);
  \draw[dashed, brown] (g) --(h);
\end{tikzpicture}
\caption{\footnotesize {\rm  Bifurcation diagram of a co-dimension one steady state bifurcation in the fundamental network of $\Sigma_4$. This figure depicts the nontrivial solution branches in case $a<0$ and $b, c>0$.}}
\label{figbif2}
\end{figure}

\subsubsection*{Bifurcations for $\Sigma_5$}
We have graphically depicted $\Sigma_5$ in Figure \ref{pict7} below.
\begin{figure}[ht]\renewcommand{\figurename}{\rm \bf \footnotesize Figure}
\centering
\begin{tikzpicture}[->,>=stealth',shorten >=1pt,auto,node distance=1.5cm,
                    thick,main node/.style={circle,draw,font=\sffamily\large\bfseries}]

  \node[main node] (1) {$x_1$};
  \node[main node] (2) [below of=1] {$x_2$};
  \node[main node] (3) [below of=2] {$x_3$};

   \node[main node] (4) [right of=1] {$x_1$};
  \node[main node] (5) [below of=4] {$x_2$};
  \node[main node] (6) [below of=5] {$x_3$};
  \node[main node] (7) [right of =4] {$x_1$};
  \node[main node] (8) [below  of=7] {$x_2$};
  \node[main node] (9) [below of=8] {$x_3$};

  \path[every node/.style={font=\sffamily\small}]
(1) edge [loop below] node {} (1)
(2) edge [loop below] node {} (2)
(3) edge [loop below] node {} (3)

(5) edge node {} (4)
(5) edge [loop below] node {} (5)
(6) edge [loop below] node {} (6)

(8) edge [loop below] node {} (8)
(9) edge [bend right] node {} (7)
(9) edge [loop below] node {} (9)
    ;

\end{tikzpicture}
\caption{\footnotesize {\rm The collection $\Sigma_5$ depicted as a directed multigraph.}}
\label{pict7}
\end{figure}
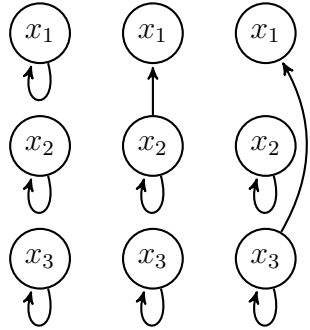

\noindent The representation of $\Sigma_5$ is given by 
  \begin{align}\nonumber
\begin{array}{rl}
A_{\sigma_1}(X) &= (X_1, X_2, X_3) \, ,\\
A_{\sigma_2}(X) &= (X_2, X_2, X_2) \, ,\\
A_{\sigma_3}(X) &= (X_3, X_3, X_3)\, .
\end{array} 
 \end{align}
 This representation nonuniquely splits as a sum of mutually nonisomorphic indecomposables
$$\{X_2=X_3=0\} \oplus \{X_1 + aX_2 =(1+a)X_3\} \ \mbox{for} \ a\in \R\, .$$ 
The maps $A_{\sigma_2}$ and $A_{\sigma_3}$ both act on the subrepresentation $\{X_2=X_3=0\}$ as the zero map, so only a transcritical bifurcation can generically occur along this irreducible representation.

Because the two-dimensional indecomposable subrepresentations $\{X_1 + aX_2 =(1+a)X_3\}$ are all isomorphic by the Krull-Schmidt theorem, let us consider only $\{X_1=X_2\}$ and choose coordinates $(X_1, X_3)$. The action of $A_{\sigma_2}$ and $A_{\sigma_3}$ on this subrepresentation reads
\begin{align}\nonumber
\begin{array}{rl}
A_{\sigma_2}(X_1, X_3) &= (X_1, X_1) \, ,\\
A_{\sigma_3}(X_1, X_3) &= (X_3, X_3)\, .
\end{array} 
 \end{align}
This confirms that $\{X_1=X_2\}$ is indecomposable, but not irreducible, because it contains the one-dimensional subrepresentation $\{X_1=X_2=X_3\}$. Moreover, one computes that 
$${\rm End}(\{X_1=X_2\})=\{ \alpha I \, |\, \alpha \in \R\}\, . $$
This shows that $\{X_1=X_2\}$ is absolutely indecomposable (i.e. of real type), and that there do not exist nontrivial nilpotent endomorphisms.

The bifurcation equation $r(X_1, X_3; \lambda) = (r_1(X_1, X_3; \lambda), r_3(X_1, X_3; \lambda)) = (0,0)$ is equivariant precisely when
\begin{align}\nonumber
&r_{1}(X_1, X_1; \lambda)= r_{3}(X_1, X_1; \lambda)=r_1(X_1, X_3; \lambda), \\ \nonumber 
&r_1(X_3, X_3;\lambda) =  r_3(X_3, X_3;\lambda) = r_3(X_1, X_3;\lambda) \, .
\end{align}
These conditions imply that
\begin{align}\nonumber
r_1(X_1, X_3; \lambda)& =  a \lambda + b X_1^2 + \mathcal{O}(|\lambda|\cdot |X_1|+ |X_1|^3)\, , \\ \nonumber 
 r_3(X_1, X_3; \lambda)& =  a \lambda + b X_3^2  + \mathcal{O}(|\lambda|\cdot |X_3|+|X_3|^3) \, .
\end{align} 
Under the generic conditions that $a, b \neq 0$, this gives two solution branches:
\begin{align}\nonumber
\begin{array}{l}
 X_1 = X_2= X_3= \pm \sqrt{-(a/b)\lambda} + \mathcal{O}(\lambda)\, ,\\
 X_1 = X_2= -X_3= \pm \sqrt{-(a/b)\lambda} + \mathcal{O}(\lambda)\, .
\end{array}
\end{align}
In this bifurcation a fully synchronous saddle-node branch and a partially synchronous saddle-node branche meet. A diagram of this bifurcation is given in Figure \ref{figbif3}.
 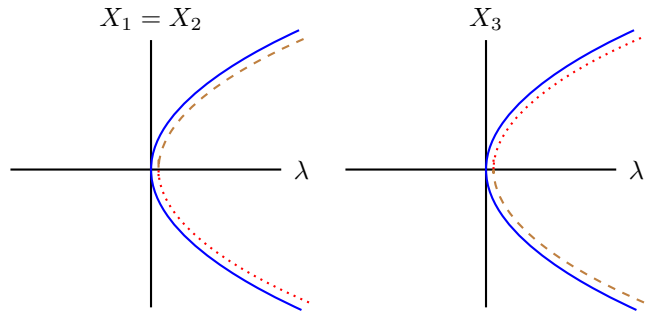
\begin{figure}[ht]\renewcommand{\figurename}{\rm \bf \footnotesize Figure}
\centering
\begin{tikzpicture} [>=stealth',shorten >=1pt,auto,node distance=3cm,
                    thick,main node/.style={draw,font=\sffamily\bfseries}]
 \node (a) at (-2,0) {};
  \node (y) at (0.1,0) {};
 \node (m) at (0,-2){};
 \node (x) at (0,2) {$X_1=X_2$};
 \node (b) at (2,0) {$\lambda$};
 \node (e) at (2,2) {};
 \node (e1) at (2.1,1.9) {};
\node (f1) at (2.1,-1.9) {};
 \node (f) at (2,-2) {};
 \node (k) at (2.1,2) {};
 \node (l) at (2.1,-2) {};
 \draw (x)--(m);
  \draw (a) --(b);
  \draw[rotate=270, blue] (f.north) parabola bend (0,0) (e.south);
  \draw[rotate=90,dashed, brown] (y.north) parabola bend (0,-0.1) (e1.south);
   \draw[rotate=90,dotted, red] (f1.north) parabola bend (0,-0.1) (y.south);\end{tikzpicture}
\begin{tikzpicture} [>=stealth',shorten >=1pt,auto,node distance=3cm,
                    thick,main node/.style={draw,font=\sffamily\bfseries}]
 \node (a) at (-2,0) {};
 \node (y) at (0.1,0) {};
  \node (m) at (0,-2){};
 \node (x) at (0,2) {$X_3$};
\node (b) at (2,0) {$\lambda$};
\node (e) at (2,2) {};
\node (e1) at (2.1,1.9) {};
\node (f) at (2,-2) {};
\node (f1) at (2.1,-1.9) {};
\node (k) at (-2,2) {};
 \node (l) at (-2,-2) {};
  \draw (x)--(m);
  \draw (a) --(b);
  \draw[rotate=270, blue] (f.north) parabola bend (0,0) (e.south);
  \draw[rotate=90,dotted, red] (y.north) parabola bend (0,-0.1) (e1.south);
   \draw[rotate=90,dashed, brown] (f1.north) parabola bend (0,-0.1) (y.south);
\end{tikzpicture}
\caption{\footnotesize {\rm Bifurcation diagram of a co-dimension one steady state bifurcation in the fundamental network of $\Sigma_5$. This figure depicts the solution branches in case $a>0$ and $b<0$.}}
\label{figbif3}
\end{figure}

\subsubsection*{Bifurcations for $\Sigma_6$}
We have graphically depicted $\Sigma_6$ in Figure \ref{pict8} below.
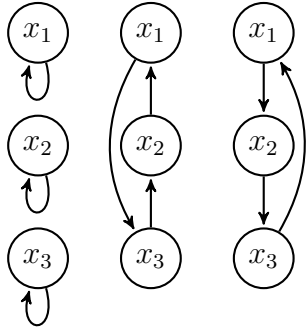
\begin{figure}[ht]\renewcommand{\figurename}{\rm \bf \footnotesize Figure}
\centering
\begin{tikzpicture}[->,>=stealth',shorten >=1pt,auto,node distance=1.5cm,
                    thick,main node/.style={circle,draw,font=\sffamily\large\bfseries}]

  \node[main node] (1) {$x_1$};
  \node[main node] (2) [below of=1] {$x_2$};
  \node[main node] (3) [below of=2] {$x_3$};

   \node[main node] (4) [right of=1] {$x_1$};
  \node[main node] (5) [below of=4] {$x_2$};
  \node[main node] (6) [below of=5] {$x_3$};
  \node[main node] (7) [right of =4] {$x_1$};
  \node[main node] (8) [below  of=7] {$x_2$};
  \node[main node] (9) [below of=8] {$x_3$};

  \path[every node/.style={font=\sffamily\small}]
(1) edge [loop below] node {} (1)
(2) edge [loop below] node {} (2)
(3) edge [loop below] node {} (3)

(5) edge node {} (4)
(6) edge node {} (5)
(4) edge [bend right] node {} (6)

(7) edge node {} (8)
(9) edge [bend right] node {} (7)
(8) edge node {} (9)
    ;

\end{tikzpicture}
\caption{\footnotesize {\rm The collection $\Sigma_6$ depicted as a directed multigraph.}}
\label{pict8}
\end{figure}
The semigroup $\Sigma_6$ is the group $\Z_3$. Only for completeness, we shall now recall some well-known facts from the bifurcation theory of $\Z_3$-equivariant differential equations.

 The representation of $\Sigma_6$ is given by 
  \begin{align}\nonumber
\begin{array}{rl}
A_{\sigma_1}(X) &= (X_1, X_2, X_3) \, ,\\
A_{\sigma_2}(X) &= (X_2, X_3, X_1) \, ,\\
A_{\sigma_3}(X) &= (X_3, X_1, X_2)\, .
\end{array} 
 \end{align}
 This representation uniquely splits as a sum of mutually nonisomorphic irreducibles
$$ \{X_1=X_2=X_3\}\oplus \{ X_1+X_2+X_3=0\} \, .$$ 
The action of $\Sigma_6$ on the subrepresentation $\{X_1=X_2=X_3\}$ is trivial, so only a saddle-node bifurcation can generically occur along this irreducible representation.

On the subrepresentation $\{X_1+X_2+X_3=0\}$, let us choose coordinates 
$$Y_1:=\sqrt{3}(X_1+X_2)\, , \, Y_2:=X_1-X_2 \, .$$ 
Then the action of $A_{\sigma_2}$ and $A_{\sigma_3}$ on this subrepresentation is given by the rotations
\begin{align}\nonumber
\begin{array}{rl}
A_{\sigma_2}(Y_1, Y_2) &= ( \cos(2\pi/3)Y_1 - \sin(2\pi/3)Y_2, \sin(2\pi/3)Y_1 + \cos(2\pi/3)Y_2) \, ,\\
A_{\sigma_3}(Y_1, Y_2) &=  ( \cos(4\pi/3)Y_1 - \sin(4\pi/3)Y_2,  \sin(4\pi/3)Y_1 + \cos(4\pi/3)Y_2) \, .
\end{array} 
 \end{align}
This confirms that $\{X_1+X_2+X_3\}$ is irreducible (over the real numbers). Moreover, one computes that 
$${\rm End}(\{X_1+X_2+X_3=0\})=\left\{ \left(\begin{array}{c} Y_1 \\ Y_2 \end{array} \right) \mapsto \left(\begin{array}{cc} \alpha & \beta \\ -\beta & \alpha \end{array} \right) \left(\begin{array}{c} Y_1 \\ Y_2 \end{array} \right)   \, |\, \alpha, \beta \in \R \right\}\, . $$
This shows that $\{X_1+X_2+X_3=0\}$ is nonabsolutely irreducible (in fact of complex type), and that there do not exist nontrivial nilpotent endomorphisms.

Generically, co-dimension one steady state bifurcations do not take place along an irreducible representation of complex type, so our bifurcation analysis of $\Sigma_6$ ends here.

\subsubsection*{Bifurcations for $\Sigma_7$}
We have graphically depicted $\Sigma_7$ in Figure \ref{pict9} below.
\begin{figure}[ht]\renewcommand{\figurename}{\rm \bf \footnotesize Figure}
\centering
\begin{tikzpicture}[->,>=stealth',shorten >=1pt,auto,node distance=1.5cm,
                    thick,main node/.style={circle,draw,font=\sffamily\large\bfseries}]

  \node[main node] (1) {$x_1$};
  \node[main node] (2) [below of=1] {$x_2$};
  \node[main node] (3) [below of=2] {$x_3$};

   \node[main node] (4) [right of=1] {$x_1$};
  \node[main node] (5) [below of=4] {$x_2$};
  \node[main node] (6) [below of=5] {$x_3$};
  \node[main node] (7) [right of =4] {$x_1$};
  \node[main node] (8) [below  of=7] {$x_2$};
  \node[main node] (9) [below of=8] {$x_3$};

  \path[every node/.style={font=\sffamily\small}]
(1) edge [loop below] node {} (1)
(2) edge [loop below] node {} (2)
(3) edge [loop below] node {} (3)

(5) edge [bend left] node {} (4)
(4) edge [bend left] node {} (5)
(6) edge [loop below] node {} (6)

(9) edge node {} (8)
(9) edge [bend right] node {} (7)
(9) edge [loop below] node {} (9)
    ;

\end{tikzpicture}
\caption{\footnotesize {\rm The collection $\Sigma_7$ depicted as a directed multigraph.}}
\label{pict9}
\end{figure}
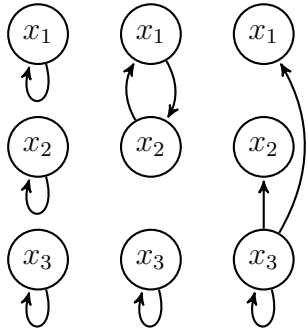

\noindent The representation of $\Sigma_7$ is given by 
  \begin{align}\nonumber
\begin{array}{rl}
A_{\sigma_1}(X) &= (X_1, X_2, X_3) \, ,\\
A_{\sigma_2}(X) &= (X_2, X_1, X_3) \, ,\\
A_{\sigma_3}(X) &= (X_3, X_3, X_3)\, .
\end{array} 
 \end{align}
 This representation uniquely splits as a sum of mutually nonisomorphic one-dimensional irreducible representations
$$\{X_1=X_2=X_3\}\oplus \{X_1=X_2, X_3=0\}\oplus \{X_1+X_2=0, X_3=0\}\, .$$ 
The action of $\Sigma_7$ on the subrepresentation $\{X_1=X_2=X_3\}$ is trivial, so only a saddle-node bifurcation can generically occur along this irreducible representation.

On the subrepresentation $\{X_1=X_2, X_3=0\}$, the map $A_{\sigma_2}$ acts as identity, while $A_{\sigma_3}$ act as the zero map. Thus one expects a transcritical bifurcation to occur along this irreducible representation.

On the subrepresentation $\{X_1+X_2=0, X_3=0\}$, the map $A_{\sigma_2}$ acts as minus identity, while $A_{\sigma_3}$ acts as the zero map. This means that a generic steady state bifurcation along this irreducible representation must be a pitchfork bifurcation.

\subsubsection*{Our running example revisited}
For our running Example \ref{running}, the composition table of $\Sigma=\{\sigma_1, \sigma_2, \sigma_3\}$ was found in Example \ref{comptable}. It turns out that this composition table is identical to that of $\Sigma_3$. As a consequence, the fundamental networks for $\Sigma$ and $\Sigma_3$ must be the same. Indeed, we computed the fundamental network of our example in Example \ref{fundamentalexample} and it coindices with that of $\Sigma_3$.

Let us assume now that the response function $f:V^3\times (\lambda_0, \lambda_1)\to V$ depends on a parameter. 
Then the differential equations of our running example become
\begin{align}\label{1}
\begin{array}{ll} \dot x_1 = & f(x_1, x_1, x_1; \lambda) \\  \dot x_2 = & f(x_2, x_2, x_1; \lambda) \\ \dot x_3 = & f(x_3, x_1, x_1; \lambda)\end{array} \, .
\end{align}
The corresponding fundamental network reads
\begin{align}\label{2}
\begin{array}{ll} \dot X_1 = & f(X_1, X_2, X_3; \lambda) \\  \dot X_2 = & f(X_2, X_2, X_3; \lambda) \\ \dot X_3 = & f(X_3, X_3, X_3; \lambda)\end{array} \, .
\end{align}
Recall that when $V=\R$, then our analysis of the fundamental network of $\Sigma_3$ predicts three possible generic co-dimension one steady state bifurcations:
\begin{itemize}
\item[i)]  A fully synchronous saddle-node bifurcation inside $\{X_1=X_2=X_3\}$. 
\item[ii)] A partially synchronous transcritical bifurcation inside $\{X_2=X_3\}$.
\item[iii)] A partially synchronous transcritical bifurcation inside $\{X_1=X_2\}$.
\end{itemize}
To understand how these scenarios impact the original network (\ref{1}), let us recall from Example \ref{conjugateexample} and Remark \ref{equilibriaremark}  that $(x_1, x_2, x_3)$ is an equilibrium point of (\ref{1}) if and only if it is mapped to an equilibrium point of (\ref{2}) by all the maps $\pi_1, \pi_2, \pi_3: V^3\to V^3$ given by 
$$\begin{array}{l} \pi_1(x_1, x_2, x_3) = (x_1, x_1, x_1)\, , \\ \pi_2(x_1, x_2, x_3) = (x_2, x_2, x_1)\, ,\\ \pi_3(x_1, x_2, x_3) = (x_3, x_1, x_1)\, . \end{array}$$
As a consequence, we find the following:

\begin{itemize}
\item[i)]  Assume that the fundamental network undergoes a fully synchronous saddle-node bifurcation. Then all its local equilibria lie inside the diagonal $\{X_1=X_2=X_3\}$. Now one can remark that $\pi_1$ always sends the point $(x_1, x_2, x_3)$ to the diagonal, but $\pi_2$ does so only if $x_1=x_2$ and $\pi_3$ only if $x_1=x_3$. Thus, the point $(x_1,x_2, x_3)$ can only be an equilibrium if $x_1=x_2=x_3$. In other words, if the fundamental network undergoes a fully synchronous saddle-node bifurcation, then so does the original network. 

\item[ii)] 
It is clear that $\pi_1$ and $\pi_3$ always map $(x_1, x_2, x_3)$ inside $\{X_2=X_3\}$ but $\pi_2$ only does so if $x_1=x_2$. Thus, if the fundamental network undergoes a partially synchronous transcritical bifurcation inside $\{X_2=X_3\}$, then the original network undergoes a partially synchronous transcritical bifurcation inside $\{x_1=x_2\}$. 

\item[iii)] 
Similarly, $\pi_1$ and $\pi_2$ always map $(x_1, x_2, x_3)$ inside $\{X_1=X_2\}$ but $\pi_3$ only does so if $x_1=x_3$. Thus, if the fundamental network undergoes a partially synchronous transcritical bifurcation inside $\{X_1=X_2\}$, then the original network undergoes a partially synchronous transcritical bifurcation inside $\{x_1=x_3\}$. 

\end{itemize}
Our message is that the monoid structure of $\Sigma$ both explains and predicts these bifurcation scenarios. Nevertheless, let us for completeness also show how they can be found from direct calculations of the steady states of (\ref{1}): 
\begin{itemize}
\item[i)] Assume that $f(0,0,0; 0)=0$ and let us Taylor expand 
$$f(X, X, X) = a \lambda + bX + c X^2 + \mathcal{O}(|\lambda|^2 +|\lambda| \cdot |X| + |X|^3)\, .$$ When $b=0$ and $a, c \neq 0$, we find as solutions of (\ref{1})
\begin{align}\nonumber
x_1=x_2=x_3 =\pm \sqrt{-(a/c)\lambda} + \mathcal{O}(|\lambda|)\, .
\end{align}

\item[ii)] Assume that $f(0,0,0;\lambda)=0$ and let us Taylor expand 
\begin{align}\nonumber
f(X_1, X_2, X_3; \lambda) & = (a + b\lambda) X_1 +c X_2 + d X_3 + e X_1^2 \\ \nonumber & + \mathcal{O}(|\lambda|^2\cdot |X_1| + |\lambda|\cdot |X_2|+ |\lambda|\cdot |X_3| + |X_2|^2+|X_3|^2+|X_1|^3) \, .
\end{align}
When $a=0$ and $b, c, c+d, e \neq 0$, we find as solutions of (\ref{1})
\begin{align}\nonumber
x_1=x_2=x_3=0 \ \mbox{and}\ x_1=x_2=0, x_3 = - (b/e)\lambda + \mathcal{O}(|\lambda|^2)\, .
\end{align}
\item[iii)] Assume that $f(0,0,0;\lambda)=0$ and let us Taylor expand 
\begin{align}\nonumber
f(X, X, Y; \lambda) = (a + b\lambda) X + cY + d X^2 + \mathcal{O}(|\lambda|^2\cdot |X| + |\lambda|\cdot |Y|+ |Y|^2+|X|^3) \, .
\end{align}
When $a=0$ and $a+c, b, d \neq 0$, we find as solutions of (\ref{1})
\begin{align}\nonumber
x_1=x_2=x_3=0 \ \mbox{and}\ x_1 = x_3 = 0, x_2= - (b/d)\lambda + \mathcal{O}(|\lambda|^2)\, .
\end{align}
\end{itemize}

\bibliography{CoupledNetworks}
\bibliographystyle{amsplain}

  \end{document}